\documentclass[a4paper,11pt]{article}
\usepackage{amsmath,amssymb,amstext,amsthm,enumerate}
\usepackage{tikz}
\usepackage{amsfonts,graphicx,latexsym}
\usepackage{xcolor}
\newcommand{\andr}[1]{{\color{blue} #1}}
\newcommand{\andre}[1]{{\color{red} #1}}

\numberwithin{equation}{section}

\newcommand{\Tcal}{\mathcal T}

\newcommand{\halmos}{\rule{1ex}{1.4ex}}
\newcommand{\eqa}{\begin{eqnarray}}
\newcommand{\ena}{\end{eqnarray}}
\newcommand{\eq}{\begin{equation}}
\newcommand{\en}{\end{equation}}
\newcommand{\eqs}{\begin{eqnarray*}}
\newcommand{\ens}{\end{eqnarray*}}
\def\qed{$\halmos$}

\def\emptyset{\varnothing}
\def\ti{\to\infty}
\def\d{\delta}
\def\e{\varepsilon}

\def\S{\Sigma}

\def\Sc{{{\rm Sys}}}

\newcommand{\R}     {\mathbb{R}}
\newcommand{\Z}     {\mathbb{Z}}
\newcommand{\N}     {\mathbb{N}}

 \newcommand{\floor}[1]{\left\lfloor #1 \right\rfloor}

\def\1{{\mathchoice {1\mskip-4mu\mathrm l}      % Blackboard bold 1
{1\mskip-4mu\mathrm l}
{1\mskip-4.5mu\mathrm l} {1\mskip-5mu\mathrm l}}}
\newcommand{\ssup}[1] {{{\scriptscriptstyle{({#1}})}}}
\def\comment#1{}

\newtheorem{theorem}{Theorem}[section]
\newtheorem{lemma}[theorem]{Lemma}
\newtheorem{proposition}[theorem]{Proposition}

\newtheorem{remark}[theorem]{Remark}
\newtheorem{definition}[theorem]{Definition}

\newtheorem{example}[theorem]{Example}

\newcommand{\heap}[2]{\genfrac{}{}{0pt}{}{#1}{#2}}

\renewcommand{\andre}{}
\renewcommand{\andr}{}
\renewcommand{\d}{{\rm d}}

\newcommand{\eps}{\varepsilon}

\newcommand{\Ccal}   {{\mathcal C }}

\newcommand{\Fcal}   {{\mathcal F }}
\newcommand{\Gcal}   {{\mathcal G }}

\newcommand{\Wcal}   {{\mathcal W }}

\renewcommand{\e}   {{\operatorname e }}

\def\ignore#1{}

\def\Def{\ :=\ }

\def\bbP{\mathbb{P}}
\def\bbE{\mathbb{E}}
\def\bbP{\mathbb{P}}
\def\bP{\mathbf{P}}

\def\giv{\,|\,}

\title{Attraction properties for general urn processes and applications to a class of  interacting  reinforced particle systems.}

\author{}
\date{}
\begin{document}

% ***************************************************************
% TITLE PAGE
% ***************************************************************

\thispagestyle{empty}
\def\thefootnote{\fnsymbol{footnote}}
\maketitle
%\vspace*{.3cm}
\begin{center}
{\large\sc Jiro Akahori\footnote{Department of Mathematical Sciences, Ritsumeikan University Email: akahori@se.ritsumei.ac.jp}},
{\large\sc Andrea Collevecchio\footnote{School of Mathematical Sciences, Monash University Email: andrea.collevecchio@monash.edu}},   {\large\sc Timothy Garoni\footnote{ARC Centre of Excellence for Mathematical and Statistical Frontiers (ACEMS) and School of Mathematical Sciences, Monash University  Email: tim.garoni@monash.edu}} and  {\large\sc Kais Hamza\footnote{School of Mathematical Sciences, Monash University Email: kais.hamza@monash.edu}}
\end{center}
\vspace*{0.1cm}
\begin{center}
(\today)\\
\end{center}
%\vspace*{0.5cm}
\noindent {\bf Abstract.}
We study   a  system of interacting  reinforced random walks defined on polygons. At each stage, each particle chooses an edge to traverse which is incident to its  position.
We allow the probability of choosing a given edge to depend on the sum of, the  number of times that particle traversed that edge, a quantity which depends on the behaviour of the other particles, and possibly external factors.
We study  localization properties of this system and our main tool is a new result we establish for a  very general class of urn models.
More specifically, we  study attraction properties of  urns composed of balls with two distinct colors which evolve as follows. At each stage a ball is extracted.  The probability of picking a ball of a certain color evolves in time. This evolution may  depend not only on the composition of the urn but also on external factors or internal ones depending on the history of the urn.
A  particular example of the latter is when the reinforcement is a function of the composition of the urn and the biggest run of consecutive picks with the same color. The model that we introduce and study is very general, and we prove that under mild conditions, one of the colors in the urn is picked only finitely often.

%\bigskip \par\noindent
%{\bf Keywords:\/} market engineering, trading protocols, competitive share, exchange market.

%\bigskip\par\noindent
%{\bf JEL Classification Numbers:\/} D51, D40, C63, C92.
%\bigskip

\noindent {\bf AMS 2000 subject classification:} 60K35
\bigskip

\noindent {\bf Keywords:} urn processes, strong reinforcement, system of reinforced random walks.

\setcounter{footnote}{0}
\def\thefootnote{\arabic{footnote}}

\newpage
\tableofcontents
%\newpage
%\addtocounter{page}{-1}

\section{Introduction}\label{intro}

\subsection{Description of the  main models and results}\label{MAIN}

This  paper contains two main results concerning an  interacting particle  system and a general urn model. The results are connected, as the latter concerns the main tool we use in our study of the former. We do however emphasize that the main result on the class of general urns is also of independent interest.

For the first result, we study a  system of interacting particles  defined on a polygon.
Each particle reinforces an edge when passing through it, and can be affected by the behaviour of all other particles as well.
More precisely, given that a particle is located on a given vertex, its next transition would be towards a neighbor of that vertex.
 The probability that it traverses a given edge is proportional to a quantity which depends on how many times that particle has previously traversed that edge, how many other particles traversed that edge, and other factors.
In order to study this system we need to keep track not only of the position of the particles but also of the transition probabilities which evolve randomly in time.   We are able to understand the main features of  the behaviour of each single particle by keeping track of its position and an auxiliary Markov chain on a larger space.  Under general assumptions we infer that if the reinforcement is strong enough,  each of the particles gets `stuck' on exactly one edge. To understand the impact on applications of this model see, for example,  \cite{OS}. In fact,
 many biological systems can be modeled by random walkers that deposit a non-diffusible signal   that
modifies the local environment for succeeding passages. The response to the environment
frequently involves movement toward or
away from an external stimulus (see \cite{OS}). In these
systems, one example of which is the motion of myxobacteria, the question arises as to whether aggregation
happens.

Our results are, to the best of our knowledge, the first involving interacting random walks with strong reinforcement. We focus on the case when the underlying graph is a polygon.
A lot of effort has been devoted to the study of single strongly reinforced particles on polygons, resulting in  a long-standing open problem to establish localization (see \cite{S08}, \cite{LimT07}, \cite{D2015}, and Section~\ref{literaturereview} for a literature review).
Moreover,  we believe that our method can, in principle, be pushed to describe the behaviour of many reinforced interacting particles on  more general graphs.

For our second main result we consider a very general urn model  where the probability to pick a ball of a given colour depends on intrinsic features of the colour, the past behavior of the urn process as well as external factors.  In particular, at each stage the composition of the urn changes by adding exactly one ball and the probability to pick a ball of a particular colour is determined by an underlying Markov process defined on a general space.
This is a significant generalization of P\'olya urns, as the transition probabilities not only may depend on the actual composition of the urn but also on the whole past and even other external factors.

Urn models have been used successfully to describe the  evolution of systems composed of elements which interact among each other. Such systems are
 of great interest in several areas of application: in medicine, clinical trials and the evolution of a system of interacting cells, in the social sciences, networks and preferential attachment (see for example \cite{Col13}), in physics and chemistry, the evolution over time of the concentration of certain molecules within cells (see for example \cite{Pem}).

For a literature review  on reinforced random walks and urns see Section~\ref{literaturereview}.
In the next section we introduce a toy model that will shed some light on the general structure we are considering in this paper.

\subsection{Example of a reinforced random walks system on a triangle with two particles}\label{motex}
\begin{figure}
\includegraphics[height=8in]{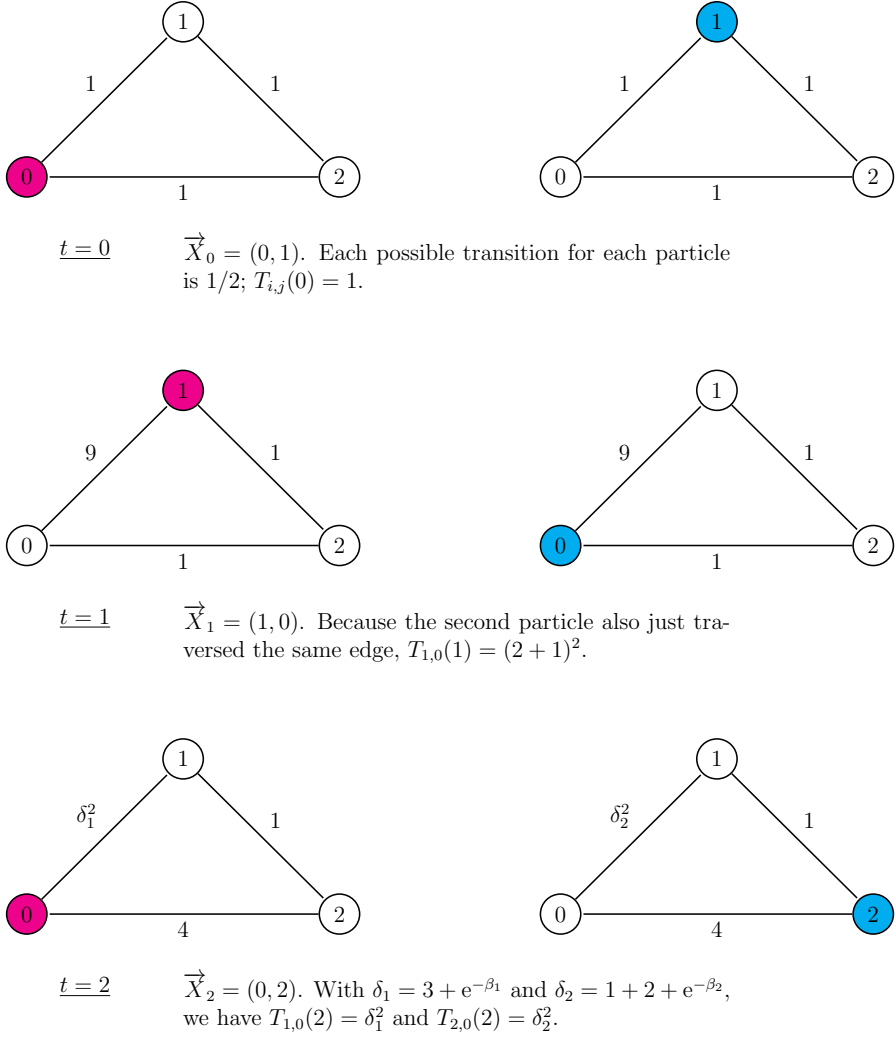}
\vspace*{-2cm}
\caption{Three steps of a system of reinforced random walks; Particle 1 is on the left and Particle 2 is on the right.}
\end{figure}
Consider a system of reinforced random walks defined as follows. The system consists of exactly two particles that take values on the vertices of a triangle labeled 0, 1 and 2. Each particle jumps at each stage to one of the other two vertices (the two nearest neighbors). We label the edges using the same set $\{0, 1, 2\}$,
where edge $j$ is the one connecting $j$ to $j+1$. All labels are understood modulo 3 so that a label of 3 equates to a label of 0.
Initially, each particle assigns a weight of one to each edge. These weights will change in time.

Starting from an initial configuration $\vec{X}_0=(X^{\ssup 1}_0,X^{\ssup 2}_0)$, we recursively define the stochastic process $\vec{{\bf X}}= \{\vec{X}_t\colon t \in \N\cup\{0\}\}$ in the following way. Let
$$ N_{i, j}(s) \Def 1+ \sum_{t=1}^s \1_{\{X^{\ssup i}_t,  X^{\ssup i}_{t-1}\}= \{j, j+1\}},$$
be the number of times particle $i$ traverses edge $j$ in either direction (plus one), and
$$ T_{i, j}(s) = \left( N_{i, j}(s) + \sum_{t=1}^s \e^{- \beta_u (s - t) } \1_{N_{u, j}(t) - N_{u, j}(t-1) = 1} \right)^2, \qquad \mbox{where $u \in \{1,2\} \setminus i$}.$$
$T_{i,j}(s)$ represents the reinforcement function for particle $i$ on edge $j$. It takes into account the times the other particle traversed edge $j$, discounted in time with rate $\beta_u>0$.

Conditional on the ``present'', $\vec{X}_s$ together with $\{T_{i, j}(s), i\in\{1, 2\}, j \in \{0 ,1 ,2\}\}$, the probability that particle $i$ jumps to vertex $j+1 \mod 3$ is given by
$$ \frac{T_{1, j}(s)}{T_{1, j}(s) + T_{1, j-1}(s)},$$
where $j-1$ is mod 3.

In this system of particles, each particle remembers its past, and detects the chemicals, e.g. slime for myxobacteria or pheromones for ants, that the other particles leave in their trails. We suppose that this chemical evaporates in time, and this explains the discount factor included in the reinforcement functions.

A particular evolution of this system up to time two is depicted in Figure 1 where the two particles were decoupled and their movements represented on separate triangles.
%The picture captures the evolution of the vectors up to time 3.
We anticipate that  each particle will visit one of the vertices only finitely often. In other words, each of the particle gets stuck on one of the edges, possibly different edges for different particles.\\
A more general example, one with $K$ particles on a general polygon, will be described in Example~\ref{exa1}. It forms part of a more general setting which is introduced in the next Section.

\subsection{Interacting reinforced random walks on finite polygons.}\label{appl}
Consider a polygon with exactly $v$ vertices, and label them using the set of  integers $V \Def \{0,1, \ldots, v-1\}$. We assume that  $j \in V$ is connected exactly to   $j+1$ and $j-1$. Here $+$ and $-$ are understood modulo $v$ and we shall simply write $j\oplus1$ and $j\ominus1$.

We denote this graph by   $\Gcal = (V, E)$, where $V$ is the set of vertices and $E$ is the set of edges.  We call $j$ the edge connecting vertex $j$ to $j\oplus1$. By this labelling, we can identify $E$ with $\{0, 1, \ldots, v-1\}$.

We consider, on a filtered probability space $(\Omega, \Wcal, \{\Wcal_s\}_{s\in\N \cup \{0\}}, \bbP)$, a system of  $K$ reinforced random walks that interact through reinforcement. More specifically, with $\N_0 \Def \N \cup \{0\}$,
denote the $\ell$-th random walk by ${\bf X}^{\ssup \ell} \Def \{X_s^{\ssup \ell}\colon s \in \N_0\}$, for $\ell \in [K] \Def \{1, 2, \ldots, K\}$ and set
$$ \vec{\mathbf{X}} \Def \Big({\bf X}^{\ssup 1}, {\bf X}^{\ssup 2}, \ldots {\bf X}^{\ssup K}\Big),\;\;
\mbox{ and } \;\;
\vec{X}_s \Def \Big( X^{\ssup 1}_s, {X}^{\ssup 2}_s, \ldots {X}^{\ssup K}_s\Big).$$

Reinforcement occurs through the number of traverses of a particular edge, augmented by an initial weight, but can also be affected by other factors, be it environmental or path-dependent.
To this end, we suppose that, for each edge $j\in E$ and each process index $\ell \in [K]$, we are given a $\Wcal_0$-measurable random variable $N_{\ell, j}(0)$, the initial weight, as well as a $\Wcal_s$-adapted process $\Xi_{\ell, j}(s)$, representing the other factors. Together, they form the reinforcement functions
\begin{equation}\label{defti}
T_{\ell, j} (s) \Def   \Big(N_{\ell, j}(s)+ \Xi_{\ell, j}(s)\Big)^\alpha
\end{equation}
for some $\alpha>1$, where 
\begin{equation}\label{enej}
\begin{aligned}
N_{\ell, j}(s)  \Def  N_{\ell, j}(0)  + \#\left\{ t \colon t \in [s],\;\; \{X_{t-1}^{\ssup \ell},  X_t^{\ssup \ell}\} =\{j, j \oplus 1 \} \right\}.
\end{aligned}
\end{equation}

Finally, each walk takes values on the vertices of $\Gcal$ and jumps, at each step, to one of two neighbors according to the rule
$$
\begin{aligned}
   \bbP(X^{\ssup \ell}_{s+1} = j  \oplus 1 \;|\; X^{\ssup \ell}_s = j, \Wcal_s) &= \frac{T_{\ell, j} (s)}{T_{\ell, j} (s) + T_{\ell, j\ominus 1 } (s)}\\
    &= 1 - \bbP(X^{\ssup \ell}_{s+1} = j \ominus 1 \;|\; X^{\ssup \ell}_s = j,\, \Wcal_s).
\end{aligned}
$$
Further, we assume that given $\Wcal_s$, the transitions of the $K$ particles at time $s+1$ are conditionally independent.

It is natural, and mathematical convenient, to restrict the choices of processes $\{\Xi_{\ell,j}(s)\}_{s\in\N_0}$ as to ensure that the process
\begin{equation}\label{madef}
M_s \Def (\vec{X}_{s},  \vec{N}_s  , \vec{T}_s ), \qquad \mbox{ for $s \in \N_0$},
\end{equation}
evolves as a strong Markov chain on a countable subset $\S$ of $V^{K} \times \N^{Kv} \times (0, \infty)^{Kv}$.

The quantities appearing in \eqref{madef} are defined as follows. For each $j \in V$ and $s \in \N_0$ let
\begin{eqnarray*}
 \vec{N}_{j} (s) &\Def& (N_{1, j} (s), N_{2, j} (s), \ldots, N_{K, j}(s)), \quad \mbox{and} \quad  \vec{N}_s \Def \big(\vec{N}_{0} (s), \vec{N}_{1} (s), \ldots, \vec{N}_{v-1} ({s})\big);\\
\vec{T}_{j} (s) &\Def& (T_{1, j} (s), T_{2, j} (s), \ldots, T_{K, j}(s)), \quad \mbox{and} \quad  \vec{T}_s \Def  \big(\vec{T}_{0} (s), \vec{T}_{1} (s), \ldots, \vec{T}_{v-1} (s)\big).
\end{eqnarray*}

Given $m\in\S$, we denote by $\bbP_m$ the probability measure on $(\Omega, \Wcal)$ that sets the initial states of the chain to $m$: $\bbP_m(M_0 = m) =1$.

 \begin{definition} With a slight abuse of notation, set $ \Sc(m) \Def  (\vec{\bf X}, \bbP_m, \{\Wcal_s\}_{s\in\N_0}, \{\Xi_{\ell,j}(s)\}_{s\in\N_0})$, which is the system of particles just described with starting point $m \in\andre{\S}$. We denote by $ \Sc = \{\Sc(m)\colon m \in \andre{\S}\}$.
 \end{definition}

 \begin{theorem}\label{cyclethm}
Consider $\Sc(m)$ with a fixed $m \in \S$.  Assume that $\sup_{\ell, j, s}\Xi_{\ell, j}(s) \le R$, a.s.,  for some constant $R$.
Then  each $ {\bf X}^{\ssup \ell}$,  with $\ell \in [K]$, will localize on one edge, i.e.   for each $\ell \in [K]$,  there exists a vertex $j_\ell \in V$ such that  for all  large $k $ we have that $X^{\ssup \ell}_k \in \{j_\ell, j_\ell \oplus 1\}$.
\end{theorem}
Notice that the model is only well defined for a finite number of particles. This circumvents the difficulty of dealing with infinite products.
\begin{example}\label{exa1} \andre{Define  $\Xi_{\ell, j}(s)$  as follows. }
$$  \Xi_{\ell, j}(s) = \sum_{\heap{\kappa \in [K] }{\kappa \neq \ell}} \sum_{t = 1}^s \e^{- \beta (s - t) } \1_{N_{\kappa, j}(t) - N_{\kappa, j}(t-1) = 1},
$$
\andre{for a fixed $\beta>0$, $\ell \in [K]$, $j  \in E$ and $s \in \N_0$}.
In this case, the  process $\bf{X}^{\ssup \ell}$  is influenced by the number of times the other particles passed a given edge, discounted in time.
To see why $\{M_s\}_{s \in \N_0}$ is  a Markov chain in this case, notice that
$$ \Xi_{\ell, j}(s+1) = \e^{- \beta} \Xi_{\ell, j}(s) + \sum_{\heap{\kappa \in [K] }{\kappa \neq \ell}}  \1_{N_{\kappa, j}(s+1) - N_{\kappa, j}(s) = 1}.$$
Moreover,
$$ \Xi_{\ell, j}(s) \le \frac{\e^\beta K}{\e^\beta -1}.$$
In the case of $K=2$, we can allow different discount factors $\beta_\kappa$, as described in the Section~\ref{motex}. In fact, the markovian structure is preserved.

In this example, particles enjoy a support system with diminishing memory. Each time a particle traverses an edge, the likelihood that it is traversed by other particles is increased. However, this effect diminishes (exponentially) over time. An example of such behaviour could be found in ants. As they travel along a path, they deposit pheromone, a behaviour-altering chemical agent that encourages other ants to take the same path. This chemical evaporates over time reducing its ability to reinforces the path. The system reinforces recently visited edges.

\end{example}
\begin{example}
For $\ell \in [K]$, $j \in E$ and $s \in \N$, set $\tau_{\ell, j}(s)  = 0$ on the event $N_{\ell, j}(s-1) = N_{\ell, j}(0)$ and
\begin{equation}\label{randti}
 \tau_{\ell, j}(s) \Def \sup\{t \in [s-1]\colon  N_{\ell, j}(t) > N_{\ell, j}(t-1)\}
\end{equation}
otherwise.
Set
$$  \Xi_{\ell, j}(s) = \sum_{\heap{\kappa \in [K] }{\kappa \neq \ell}} \1_{\tau_{\kappa, j}(s) \ge \tau_{\ell, j}(s)}.$$

\end{example}
%\begin{remark} The proof  of Theorem~\ref{cyclethm} shows that we could consider more general  functional forms for $T_{\ell,j}(s)$ than the one given in \eqref{defti}. In particular, Theorem~\ref{cyclethm}  extends to the case
%$$ T_{\ell, j} (s) \Def g\Big(N_{\ell, j}(s)+ \Xi_{\ell, j}(s)\Big),$$
%with the deterministic function $g \colon \N \mapsto (0, \infty)$ satisfying
%\begin{itemize}
%\item $\sum_{s=1}^\infty 1/g(s) <\infty$;
%\item the function  $g$  is non-decreasing;
%\item $\begin{displaystyle}
%\liminf_{n \ti} \frac{g(n) \sum_{s=n}^\infty (1/g(s)^2)}{\sum_{s=n}^\infty (1/g(s+R))} > 0
%\end{displaystyle}$,
%where $R$  was given in Theorem~\ref{cyclethm}.
%\item There exists $R_1 \in (0, \infty)$ such that
%$\begin{displaystyle}
%\sum_{s=n+R_1}^\infty (1/g(s)^2)  \le \sum_{s=n}^\infty (1/g(s+R)^2)
%\end{displaystyle}$,
%for all large $n$.
%\end{itemize}
%For example, we can consider  $g(k) =  k^\alpha (\ln k)^\beta$ for $\alpha >1, \beta>0$, or $g(k) = \e^{ \gamma k}$, for some $\gamma >0$.
%\end{remark}
%\begin{corollary}
%If $G(k) = k^\alpha$ for some $\alpha>1$, and
%$$ \Xi_{\ell, s}(j) = \sum_{r \neq \ell}  N_{r, j}(s),$$
%for all $\ell \in [K]$
% then  each $ {\bf X}^{\ssup \ell}$, with $\ell \in [K]$,  will localize on one edge  (maybe not the same edge).
%\end{corollary}

\subsection{Definition of Generalized Urn Processes (GUP).}\label{GUP}
Before we introduce the definition of Generalized Urn Processes (GUP), we analyze a key example, where the transition probabilities of the urn depend not only on its actual composition but also on other factors.

\begin{example}\label{Example1} Consider an urn  which initially contains exactly two balls, one white and one red. At each stage, a ball is added according to a certain random rule described below. Fix an increasing function $\Psi \colon (0, \infty) \mapsto (0, \infty)$ with the following properties.  Either
\begin{itemize}
\item $\lim_{x\ti} \Psi(x)/\Psi(x-1) = \infty$, or
\item $\Psi$  is twice  continuosly differentiable, $\int_{1}^\infty (1/\Psi(u)) {\rm d}u <\infty$, $\liminf_{x \ti} \Psi'(x)>0$,  $\lim_{x\ti} \Psi(x)/\Psi(x-1) $ exists in $[1, \infty)$, and 
$$ \lim_{x \ti} \frac{\Psi''(x) \Psi(x)}{\Psi'(x)^2}  \mbox{ exists in } (0 ,\infty].
$$
\end{itemize}
Let $\{Z_n\}_{n \in \N_0}$ be a homogeneous  Markov chain taking values in a countable state space and adapted to a given filtration $\{\Gcal_n\}_n$.
Denote the composition of the urn at time $n$ as follows. Set $N_w(n)$ (resp. $N_r(n)$)  the number of  white  (resp.  red) balls in the urn at time $n$,   with $N_w(n) +N_r(n) = n+2$. Then, conditional on $\Gcal_n$, the probability to pick a white ball at stage $n+1$ is
$$ \frac{\Psi(N_w(n)+g_1(Z_n))}{\Psi(N_w(n)+g_1(Z_n)) + \Psi(N_r(n)+g_2(Z_n))},$$
where $g_1$ and $g_2$ are two given positive functions. Note that $\{\Gcal_n\}_n$ is possibly larger than the natural filtration of the urn process. We  assume that $g_i(Z_n) \le \theta(n+2)$,  a.s. for all large $n$, for $i \in \{1, 2\}$, and some positive increasing function $\theta$ satisfying
\begin{equation}\label{prot}
\frac{\Psi(z + a\ln z )}{\Psi\big(z + \theta(z)\big)} \ge  1 + \frac 1{2z}, \qquad \mbox{for some $a \in (0,1/2)$ and for all large $z$}.
\end{equation}
We prove (see Theorem~\ref{Examp1} below) that, under the above assumptions the urn localizes on one of the two colours. More precisely one of the colors will be picked only finitely often. This a by-product of a more general result, i.e. Theorem~\ref{rubin}. Examples of functions which satisfy the properties above are $\Psi_1(x) =  x^\alpha$, with $\alpha>1$,  $\Psi_2(x) = x^\alpha (\ln x)^\beta$ for $\alpha >1, \beta>0$,   and $\Psi_3(x) = \e^{ \gamma x}$, for some $\gamma >0$.  For these choices of $\Psi$,  we can choose $\theta(x)= a' \ln n$ for any $a' \in (0, 1/2).$
\end{example}

 Define a generalized urn process, called GUP$(\mathbf{M}, M_0 =m, f_1, f_2)$, as follows. Informally, the key in this definition is that there is a driving Markov process, defined on a larger abstract space, which carries the information about the transition probabilities and the composition of the urn.

Formally,   consider an urn which initially contains a fixed number of balls,  $\phi_{1}(0)$ white and $\phi_{2}(0)$ red, where $\phi_{1}(0), \phi_{2}(0) \in \N$.  Suppose that at each step a ball is picked, its colour observed, and returned to the urn together with another ball of the same colour. Given the history of the process up to time $k-1$, the probability of picking a white ball is
\begin{equation}\label{prob}
\frac{f_1(k)}{f_1(k)+f_2(k)},
\end{equation}
where $f_{1}(k)$ and $f_{2}(k)$  are positive, random and depend on the history of the process up to time $k-1$.  We set $\vec{f}(k) = (f_{1}(k), f_{2}(k))$, for $k\in \N$.
Denote by $\vec{\phi}(k)= (\phi_1(k), \phi_2(k))$, the composition of the urn at time $k$; $\phi_1(k)$ white balls and $\phi_2(k)$ red balls.
We assume that the urn process satisfies the following properties.
\begin{itemize}
\item [A)] There exists a countable  space $\Sigma$, such that for each $m \in \Sigma$ there exists a probability space $(\Omega, \Fcal, \bbP_m)$
and a  homogeneous Markov chain $\mathbf{M} \Def\{M_{k}, k \ge 0\}$ on $\S$,  which satisfies the following.  For each $m \in \Sigma$, we have $\bbP_m(M_0 = m) =1$.  Moreover $\vec{f}(k+1) = G_1(M_k)$ and $\vec{\phi}(k) = G_2(M_k)$  for some functions $ G_1$ and $G_2$, with $k \in \N_0$.
%\item [B)] The counter-image $\Phi^{-1}(\{j,k\})$ is countable for all $(j,k) \in \N^2$.
\item [B)]There exists $\eps>0$ such that $\bbP_m(f_i(k) >\eps) =1$ for all $m \in \S$, $ i \in \{1, 2\}$ and $k \in \N$.
\end{itemize}
The process GUP$({\bf M},  M_0 =m,  f_{1}, f_{2})$, described above,  is a strong generalization of P\'olya urns. Our aim is to understand the behaviour of the urn based on the  behaviour of each of the  $f_i(k)$, with $i\in \{1, 2\}$, $k\in \N$. The processes $f_{i}(\cdot)$ are called the reinforcement functions associated to the GUP. We emphasize that our results below only assume the existence of general quantities such as the process  $M_k$ and not their knowledge.

%
%\begin{remark}\label{uno}
%For each $m \in \S$, if we consider GUP$({\bf M},  M_0 =m,  f_{1}, f_{2})$,  the support of $\vec{\phi}(n)$ is countable, we have that the support of $M_n$ is finite. Hence, the support $f_i(n)$, under $\bbP_m$ is finite. This observation plays a key role in the next definition.
%\end{remark}
 \noindent In what follows $(x_n)$ denotes a sequence, whereas $\{x_n\}$ denote a set. The difference is that the sequence allows repetitions. We denote by $x \wedge y$ the minimum between $x$  and $y$ and by
$x \vee y$ the maximum between $x$  and $y$.
For any countable set $A \subset [\eps, \infty)$ define
$\langle A \rangle \Def \sum_{t \in A}\frac 1{t},$ with the usual convention that the sum over the empy set is zero. For each $m \in \S$, define the  sequence  of random times where we pick the \lq$i$-th colour\rq, as follows. For $ i\in \{1,2\}$, let   $u_i(1) \Def \min\{ \ell \ge 1 \colon \phi_i(\ell+1) - \phi_i(\ell)>0\}$. Suppose we defined $u_i(k)$, then let $u_i(k+1) \Def \min\{ \ell > u_i(k)  \colon \phi_i(\ell+1) - \phi_i(\ell)>0\}$.
Let
$$d_i(k, m) \Def \sup\{ c \colon \bbP_m\big(f_i(u_i(k)) < c\big) = 0 \} = {\rm essinf} f_i(u_i(k)),$$
and set $ B_{i}(m) = (d_i(k, m) \colon  k \in \N).$  Recall that $B_i(m)$ can list the same element more than once.  Note that the quantities $d_i(k, m)$ and $\langle  B_{i}(m) \rangle$, with $i \in \{1, 2\}$, $k\in \N$ and $m\in \S$,  are deterministic.

\noindent{\bf Assumption I.}
 There exists $i \in \{1, 2\}$ such that $ \langle  B_{i}(m) \rangle< \infty
 $, for all $ m \in \S$.

\begin{remark} Assumption I becomes a familiar condition in reinforced random walks and generalized P\'olya urns when the reinforcement functions only depend on the actual composition of the urn; more specifically when   \eqref{prob} holds for
$ f_i (n) = g_i(\phi_i(n))$, with $ i\in \{1, 2\}$, for a pair of deterministic functions $g_i \colon \N \mapsto (0, \infty)$. In this case,
$\langle  B_{i}(m) \rangle = \sum_{k=\phi_i(0)+1}^\infty (1/g_i(k))$  (as $\phi_i(u_i(k)) = \phi_i(0) + k$) and Assumption I is necessary and sufficient, according to the celebrated Rubin's Theorem (see \cite{BD1990}), for one colour to be picked finitely often, a.s..
\end{remark}
Denote by $A_{\infty}$ the event that $ \lim_{k \ti} \phi_1(k) = \lim_{k \ti} \phi_2(k) = \infty$. For $n \in \N$ and $m \in \Sigma$,  set
 $\alpha(m) =  \bbP_m\big(A_\infty\big)$.
The  following Proposition  is a corollary of Proposition~\ref{prap} which can be found in the Appendix.
  \begin{proposition}\label{ref01} $\sup_{m \in \Sigma}\alpha(m) \in \{0, 1\}$.
 \end{proposition}

 \begin{definition}
For any fixed sequence  $m_n \in \S$, with $n \in \N$, consider independent   GUP(${\bf M}^{\ssup n}, M^{\ssup n}_0 = m_n, f_1, f_2$),
where  \andre{(${\bf M}^{\ssup n}$, with $n \ge 1$)} denotes a sequence of   independent  copies of ${\bf M}$ with different starting points. \andre{The process ${\bf M}^{\ssup n}$ starts from $m_n$}.
Let  $\bigotimes_n \bbP_{m_n}$ be  the corresponding product measure.
 For fixed $n \in \N$,  define $\big(f_{i}(k, n)\colon   k \in \N, i \in \{1, 2\}\big)$ (resp. $(\vec{\phi}(k, n) = (\phi_1(k, n), \phi_2(k, n))\colon k \in \N)$), to be the
 sequence  of   reinforcement functions  determined by ${\bf M}^{\ssup n}$ (resp. the  sequence of composition of the $n$-th urn by time $k$). Let   $A_{\infty}(n)$  be the event that $ \lim_{k \ti} \phi_1(k, n) = \lim_{k \ti} \phi_2(k, n) = \infty$.
 Set
\begin{equation}\label{updef}
  {\rm UP}(i,n) = \andre{(}k \colon \phi_{i}(k+1, n) - \phi_{i}(k,n)>0\andre{)}.
  \end{equation}
  Let
 $$
   \Tcal_{i}(n) =  \big(f^{2}_{i}(k,n) \colon  k \in  {\rm UP}(i,n)\big).
$$
The sequences $\Tcal_{i}(n)$, with $i\in \{1, 2\}$, can be finite.

 \end{definition}

\noindent{\bf Assumption II.} Suppose $\alpha(m_{0}) >0$ and there exists a sequence $m_{n}$ that satisfies the following conditions
\begin{itemize}
\item [i)]$\sum_{n=1}^{\infty}\big(1-\alpha(m_{n})\big) < \infty$;
\item [ii)] Either $\lim_{n \ti} \langle \Tcal_1(m_n) \rangle \vee \langle  \Tcal_2(m_n)\rangle = \infty$, a.s., or
\begin{equation}\label{liminf} 
 \liminf_{n \ti}  \eta_n \frac{\langle \Tcal_1(m_n) \rangle \vee \langle  \Tcal_2(m_n)\rangle}{  \langle B_{1}(m_n)\rangle \vee \langle B_{2}(m_n)\rangle}  >0, \qquad \bigotimes_n \bbP_{m_n}-\mbox{a.s..,}
\end{equation}
where  $\eta_{n} \Def \inf \{B_{1}(m_n) \cup B_{2}(m_n)\}$.
%In \eqref{liminf} we convey $\infty/\infty =\infty$, in the sense that if both numerator and denominator are infinite for all  but finitely many $n \in \N$, then the condition \eqref{liminf} is satisfied.
\item [iii)]    There exist  two  families of random sequence\andre{s} $a_n(k) = a(k, m_n)$ and $b_n(k) =b(k, m_n)$ which satisfy the following.
 \begin{itemize}
 \item  For  $k, n \in \N$,
 $a_n(k)$ is measurable with respect to $\sigma(M^{\ssup n}_0, M^{\ssup n}_1, \ldots, M^{\ssup n}_{k-1})$,  $ a_n(k) \in \{1,2\}$, $b_n(k) \Def \{1, 2\} \setminus a_n(k)$, and
 \item Define $U^{\ssup n}_a \Def \{ k \ge 1 \colon \phi_{a_n(k)}(k+1, n) > \phi_{a_n(k)}(k, n)\}$, and $U^{\ssup n}_b \Def \N\setminus U^{\ssup n}_a$. For all $n$ larger than an a.s. finite random time (not necessarily stopping time),  we have
 \begin{equation}\label{ainp}
   \left(\sum_{i \in U^{\ssup n}_a} \frac 1{f_{a_n(i)}(i, n)} -  \sum_{j \in U^{\ssup n}_b}\frac 1{f_{b_n(j)} (j, n)}
\right)  \1_{A_\infty(n)} \le 0, \qquad \bigotimes_n \bbP_{m_n}-\mbox{a.s..}
\end{equation}
\end{itemize}
\end{itemize}
Our main result is the following.

\begin{theorem}\label{rubin}
 Consider GUP(${ \bf M}, M_{0} = m,  f_{1}, f_{2})$, for some $m\in \S$.
Assumptions I  and II cannot simultaneously hold.
 \end{theorem}

\begin{theorem}\label{Examp1} The urn described in Example \ref{Example1} is GUP and one of the colours is picked only finitely often.
\end{theorem}
 Theorem~\ref{rubin} is a very general and powerful tool to decide whether $\sup_{m \in \S} \alpha(m)= 0$ for urns where the reinforcement depends not only on the composition of the urn but also on external factors (e.g. interacting urns) or internal ones depending on the history of the urn. A very particular example is when the reinforcement is a function of the composition of the urn and the biggest run of consecutive picks with same colour. To the authors' knowledge very little is found in the literature about urns with strong dependence on the past. In general the reinforcement only depends on the composition of the urn at a given stage,  revealing still a Markovian structure. This is of course quite a constraint for the applications. Think of a consumer that has to choose between two products. He will not only observe how many times in the past one of the products has been chosen over the other by other consumers, but also the actual trend. In fact it is possible that one of the products might have been chosen less times than the other but was quite successful in recent times, attracting the preference of the consumer.

    Urn processes have a variety of applications and are very important tools in statistics, medicine and network theory (see preferential attachment models).
%We discuss few general examples involving interacting urns and general reinforcement in Section~\ref{appl}.
We emphasize that we cannot recover these results from the standard general techniques already available for generalized P\'olya urns. In particular  we were not able to apply the decoupling method introduced by Rubin (see \cite{BD1990})   in our setting to prove attraction properties of one colour over the other.  A review of the existing results and applications of urn models is given in  Section~\ref{literaturereview}.

\subsection{Literature review}\label{literaturereview}
Reinforced random walks were introduced by D. Coppersmith  and P. Diaconis   in an unpublished manuscript (\cite {CoDia}). The case of a single particle which moves to nearest neighbors vertices of a graph is considered. The probability to traverse a given edge incidental to the actual position is proportional to the weight of that edge, which in turn evolves in time as follows. Each time is traversed the edge's weight is increased by a fixed  constant $\delta>0$. This is the linear case.
A more general case was studied by B. Davis \cite{BD1990}. He studied the case where the probability to pass an edge is proportional to a reinforcement function $f$ of the times that edge has been traversed in the past. 
T. Sellke (see \cite{S94})  conjectured that if the reinforcement function $f(i)$  is strong, i.e. satisfies $\sum_j 1/(f(j))<\infty$, then RRW on the triangle visits exactly two adjacent vertices  infinitely often;  that is oscillates on one edge at all large times. V. Limic proved the result when $f(i) = i^\alpha$ for  $\alpha>1$. Limic and P. Tarr\'es (see \cite{LimT08})  extended this result on general graphs for a large class of reinforcement functions, including all the monotone increasing functions. C. Cotar and D. Thacker have recently announced (see \cite{D2015}) a complete solution of the problem on any bounded degree graph. They also include in their paper complete results for the superlinear vertex reinforced random walk.
As for the latter properties, it is curious that it shows localization even in the linear case, where each vertex is reinforced by one each time it is visited. It was proved by T\'arres (see \cite{Ta04}) that when defined on $\Z$  it localizes on exactly 5 consecutive  vertices.   See also \cite{Vo01} for the study of this process on general graph.
For a survey on reinforcement, including discussion on linear and sublinear reinforced random walk see \cite{Pem} and \cite{Koz}.
The results listed above only describe the behaviour of a single particle. In \cite{Kov08} it is  considered a system of two particles with linear reinforcement, defined on $\Z$. In that paper Y. Kovchegov proved that the two particles meet infinitely often. Yilei Hu studied the case with $K \ge 2$ particles in his Phd thesis (see \cite{Hu10}).

  To the best of our knowledge, the result contained in our paper  is the first concerning a  system of  particles which interact through strong reinforcement.
   We also emphasize the fact that our results on urn and  system of reinforced particles go beyond the simple composition of the urn or the number of times an edge has been visited. It actually covers more sophisticated transition probabilities.

\subsection{Brief description of the proof of Theorem~\ref{rubin}}
We reason by contradiction.
 We suppose that both   Assumptions I and  II hold.
From now on $m_{n}$ is used to denote a sequence which satisfies Assumption II. Later, we will work  with  subsequences of this sequence.   We embed infinitely many urns using the same Brownian motion, sequentially.
Each embedding can be  briefly described as follows. Suppose we defined $M_k$, then the composition of the urn at time $k+1$, i.e. $\vec{\phi}(k+1)$ is determined by the Brownian motion. We then use an external randomization to get $M_{k+1}$ using the fact that given $M_k$ and  $\vec{\phi}(k+1)$ the support of $M_{k+1}$ is countable and conditionally independent of the past $M_j$ with $j< k$.
Recall that we denoted by $\mathbf{M}^{\ssup n}= \{ M^{\ssup n}_k, k \ge 0\}$ the Markov processes associated to the $n$-th urn. The $n$-th urn has initial condition $M^{\ssup n}_0= m_n$. Time $S_n$, properly defined later,  will denote the time when the first $n$ urns are embedded.  We prove that $S_n<\infty$ a.s.. In fact, this stopping time is smaller than the hitting time of  a the set  $\{-C_n, C_n\}$ for some finite $C_n$. Hence the process $ W_{S_n}$, with $n \in \N$,  is a martingale.  We prove that there exists a random time $N$ such that $W_{S_n} - W_{S_{n-1}} \le 0$ for all $n \ge N$.  We also prove that under  Assumptions I and II the stopping time \underline{$S_\infty = \lim_n S_n$ has infinite first moment.}
 At this point we distinguish two cases.
 \begin{itemize}
 \item If $W_{S_\infty} >-\infty$, using the Kolmogorov three series Theorem and Burkholder-Davis-Gundy (BDG) inequality, we prove that $\mathbb{E}[S_\infty]<\infty$ which yields a contradiction \andre{with the above statement}.
 \item $W_{S_\infty} =-\infty$ then we use BDG inequality to argue that the martingale $W_{S_n}$, with $n \in \N$,  is bounded in $L^1$. This contradicts, using the Doob martingale convergence Theorem, its convergence to $-\infty$.
\end{itemize}

Throughout the embedding,   we make use of the following fact.
\begin{remark}\label{GUP}
 For any  function $f$, and any positive constant $\lambda$, denote by $\lambda f$  the function which maps $j \in \N \mapsto \lambda f(j)$. Fix $m \in \S$.  It is immediate to realize  that GUP(${ \bf M}, M_{0}= m,  \lambda f_1, \lambda f_2$), where $\lambda$ is a positive constant, has the same distribution as GUP(${ \bf M}, M_{0}= m, f_1, f_2$). In fact, the transition probabilities described in \eqref{prob} remain unchanged.
\end{remark}

%\begin{example} Consider a finite collection $ \{p_{i} \colon p_{i} \in (1, \infty), i \le m\}$. The reinforcement family  $h(n) = n^{p_{i}}$,  with $ p >1$, is good.
%\end{example}

 \subsection{Embedding of the GUP into Brownian motion}
Let $(m_n, n\in \N)$ be a sequence in $\S$  satisfying the conditions in Assumption II.
We embed GUP(${ \bf M}, M_{0}= m_1, f_1, f_2$) into Brownian motion.  Recall that we denote the reinforcement processes for this urn using $f_i(k) =f_i(k,1)$, with $i \in \{1, 2\}$. To simplify the notation, we drop the index 1 in the remaining part of this subsection.
  Fix $m_1 \in \S$. Set $\vec{\phi}(0) = \vec{\phi}(0, 1) = (\phi_1(0), \phi_2(0))$, i.e. the initial configuration of the urn, which is determined by $M_0 =\andre{m_1}$. Recall that the reinforcement functions at time 1, i.e., $f_1(1)$ and $f_2(1)$ are also determined by $M_0 = \andre{m_1}$. Set $e_1 = (1, 0)$ and $e_2 = (0, 1)$. Recall the definition of the sequences $a(k) = a_1(k)$ and $b(k) = b_1(k)$, with $k \in \N$, from Assumption II.
   Let the process $\mathbf{W} := \{W_t, \, t\ge 0\}$ be a standard Brownian motion,  which  starts at $0$. Denote by  $\{\Fcal_t \colon t \ge 0\}$  the natural filtration of this Brownian motion.  We  use this process to generate GUP(${\bf M}, M_{0} = \andre{m_1},$  $f_1, f_2$).  % $$y(a_i, \, a_k,\, \{a_{i+j}\}_{j \ge 1}, \, \{a_{k+s}\}_{s \ge 1}).$$
  % In the next Section, we use the same variables, but we show the indices.
   Set $T_{0}(1) =0$, a.s.. Let
\begin{equation}\label{ennezero}
 T_{1}(1) \Def \inf \Big\{ t \ge 0: \; W_t \mbox{ hits either } 1/f_{a(1)}(1) \mbox{ or } - 1/f_{b(1)}(1)  \Big\}.
\end{equation}
If $ W_{T_{1}(1)} - W_{T_{0}(1)} =  1/f_{a(1)}(1) $ then set $\vec{\phi}(1) = e_{a(1)} + \vec{\phi}(0)$, otherwise set $\vec{\phi}(1) = e_{b(1)} + \vec{\phi}(0)$.
Given $\vec{\phi}(1) = d$, with  $d\in \N^2$, we define $M_1$ by  generating a  random variable  $R_1$   independent of $\Fcal_{T_1(1)}$, with distribution
 $$\bbP_{\andre{m_1}}(M_{1} \in \cdot \;|\;  \vec{\phi}(1) = d),$$
 and setting $M_1= R_1$.
Suppose we defined $T_{k}(1)$ and $M_{k}$. Set $\phi_s(k)$ to be the $s$-th  coordinate of $\vec{\phi}(k)$, with $s \in \{1, 2\}$.
Set
$$ T_{k+1}(1) = \inf \left\{  t \ge T_{k}(1): \:  W_t - W_{T_{k}(1)} \mbox{ hits either }  \frac{1}{f_{a(k+1)}(k+1)} \mbox{ or }  -\frac{1}{f_{b(k+1)}(k+1) } \right\}.$$
If $ W_{T_{k+1}(1)} - W_{T_{k}(1)} =  1/f_{a(k+1)}(k+1) $ then set $\vec{\phi}(k+1) = \vec{\phi}(k)  +e_{a(k+1)} $, otherwise set $\vec{\phi}(k+1) = \vec{\phi}(k) +  e_{b(k+1)}$.
Next, given the events $\{M_k = c\} \cap \{\vec{\phi}(k+1) = d\}$, with $c \in \S, d \in \N^2$,  we define $M_{k+1}$ by generating a random variable  $R_{k+1}$, independent from  both $\Fcal_{T_{k+1}(1)}$ and $\sigma(R_1, R_2, \ldots, R_k)$,
and  with distribution
$$\bbP_{\andre{m_1}}(M_{k+1} \in \cdot \;|\; M_{k}=c, \vec{\phi}(k+1)=d),$$
and setting $M_{k+1}= R_{k+1}$.
We assume that Brownian motion and the random variables $R_k$, $k \in \N$, are defined  on  the same probability space $(\Omega, \Fcal, \bbP)$.

\noindent By the ruin problem for Brownian motion, we have that
 \[
 \begin{aligned}
 \bbP\left( \vec{\phi}(k+1) = \vec{\phi}(k) +e_{1}  \;|\;  M_{k} \right) &= \frac{  1/f_2(k+1)}{ (1/f_1(k+1)) + (1/f_2(k+1))}\\
  &= \frac{f_1(k+1)}{f_1(k+1) + f_2(k+1)}, \qquad \qquad \bbP-\mbox{a.s.,}
  \end{aligned}
  \]
 which is exactly the urn transition probability.
 \begin{remark}It is convenient for us to embed the process $\vec{\phi}(k)$ into the Brownian motion, because we use the well known results regarding this process to get a contradiction.
 \end{remark}

% Recall that we interpret $\vec{\phi}(k) $ as follows. Its first coordinate (resp. second) $\phi_{1}(k)$ (resp. $\phi_{2}(k)$) is the number  white balls (resp. red) in the urn by time $k$.

Define
\begin{equation}\label{enne}
T_{\infty}(1) \Def \lim_{k \to \infty} T_k(1).
\end{equation}
This limit exists because the sequence of stopping times $\{T_{k}(1)\}$ is increasing. For this reason $T_{\infty}(1)$ is itself a stopping time, possibly infinite.

%\begin{remark}\label{rimp}
%Notice that  Proposition~\ref{sign} holds for arbitrary initial condition $m \in \S$.
%\end{remark}
\subsection{Proof of Theorem \ref{rubin} in the case $\langle B_{1}(m_n)\rangle \vee \langle B_{2}(m_n)\rangle <\infty$, for infinitely many $ n \in \N$.}
We reason by contradiction and assume that  both Assumptions I and II hold. We further assume in this part of the proof that $\langle B_{1}(m_n)\rangle \vee \langle B_{2}(m_n)\rangle <\infty$ for  infinitely many  $n \in \N$. By taking a subsequence, we assume that  $\langle B_{1}(m_n)\rangle \vee \langle B_{2}(m_n)\rangle <\infty$ holds for  all $n \in \N$.
In virtue of Proposition~\ref{ref01}, we have that
% and let $\Tcal_{k}(g) \Def \sum_{j=k}^{\infty}1/g(j)$.
 $\lim_{n\ti}\alpha(m_n)  =1$.
% By taking subsequences, we may also assume that  $m_{n}$ satisfies
%\begin{equation}\label{mconh}
%\sum_{j=n}^{\infty} \left(1- \alpha(m_{j})\right) < \e^{-n}.
%\end{equation}
Recall the definition of  $\eta_n $ from Section~\ref{MAIN}. Set
\begin{equation}\label{cienne}
  c_n = \sqrt{  \frac{{2}}{\eta_{n}} (\langle B_{1}(m_n)\rangle \vee \langle B_{2}(m_n)\rangle)}.
  \end{equation}
  The sequence $c_n$ is deterministic, this will play a major role in what follows.
Notice that we are going to repeat the procedure of taking subsequences few more times in the remaining part of the proof.
 Consider the infinite sequence of independent copies of ${\bf M}$, say ${\bf M}^{
 \ssup 1}, {\bf M}^{\ssup 2} \ldots$ with starting points $m_1, m_2, \ldots$.
Set $S_{1} = T_{\infty}(1)$. Next, we use the process  $\{W_{S_{1} + t} - W_{S_{1}}, t \ge 0\}$ to embed GUP(${\bf M}^{\ssup 1}$, $ M_{0}^{\ssup 1} = m_{2}$,  $ c_{2} f_{1}(\cdot, 2)$,   $c_{2} f_{2}(\cdot, 2)$) defining the hitting times $T_{0}(2) = S_{1}$  and $ T_{i}(2)$ to be times when the $i$-th ball  is generated (similarly to what we have \andre{described} above) and let $S_{2}$ to be its ending stopping time, i.e. $S_{2} =  T_{\infty}(2)$.
 We repeat this procedure, to  define  the stopping times $ S_{3}, \ldots S_{k} \ldots$, each time using a new reinforcement scheme, as follows. For $k \ge 2$. we have that time $S_{k}$ is the ending time of the embedding of GUP(${\bf M}^{\ssup 1}$, $M^{\ssup 1}_{0} = m_{1},  f_{1}(\cdot, 1)$, $ f_{2}(\cdot, 1)$) and all the GUP(${\bf M}^{\ssup n}$, $M^{\ssup n}_{0} = m_{n},  c_{n}f_{1}(\cdot, n)$,   $c_{n} f_{2}(\cdot, n)$), with  $2 \le n \le k$.

\begin{remark}
The GUPs embedded are independent of each other.  As before, we assume that the Brownian motion and the variables used in the embedding are defined on the same probability space $(\Omega, \Fcal, \bbP)$. Finally, notice that by the fact that $(c_n, n \in \N)$ is a deterministic sequence, combined with  remark~\ref{GUP}, we have that GUP(${\bf M}^{\ssup n}$, $M^{\ssup n}_{0} = m_{n},  c_{n}f_{1}(\cdot, n), c_{n} f_{2}(\cdot, n))$ has same distribution as GUP(${\bf M}^{\ssup n}$, $M^{\ssup n}_{0} = m_{n},  f_{1}(\cdot, n),  f_{2}(\cdot, n))$. We use $c_n$ because it changes the time needed to embed. In this way the total time  needed to embed all the GUP's has infinite expectation, as we prove below.
\end{remark}

\begin{definition}
Set $A_\infty(n)$ to be the event that infinitely many balls of each colour are picked in the GUP(${\bf M}^{\ssup n}, M^{\ssup n}_0 = m_n, c_{n}f_{1}(\cdot, n), c_{n}f_{2}(\cdot, n)$).
\end{definition}

%As proposition \ref{sign} holds for any initial state $m$ we have the following immediate extension.
%\begin{corollary} \label{sign1}
%
%\end{corollary}
 \begin{proposition}\label{sign} There exists an a.s.  finite random time $N_1$ such that for all $n \ge N_1$, we have
 $$ \Big(W_{S_n} - W_{S_{n-1}}\Big)\1_{A_\infty(n)} \le 0, \qquad \mbox{eventually } \bbP-\mbox{a.s..}$$
 \end{proposition}
 \begin{proof}
 Recall the definitions of $U^{\ssup n}_a$ and $U^{\ssup n}_b$ from the introduction.

As $U^{\ssup n}_a$ and $U^{\ssup n}_b$ form a partition of $\N$, we have that
$$
\begin{aligned}
(W_{S_n} - W_{S_{n-1}})\1_{A_\infty(n)}  &= \1_{A_\infty(n)} \sum_{j=1}^\infty (W_{T_j(n)} - W_{T_{j-1}(n)})\\
 &=  \left(\sum_{w \in U^{\ssup n}_a} (W_{T_{w+1}(n)} - W_{T_{w}(n)})   + \sum_{r \in U^{\ssup n}_b} (W_{T_{r+1}(n)} - W_{T_{r}(n)})\right) \1_{A_\infty(n)}\\
 &=  \frac{1}{c_n} \left(\sum_{w \in U^{\ssup n}_a} \frac 1{f_{a_n(w)}(w, n)} -  \sum_{r \in U^{\ssup n}_b} \frac 1{f_{b_n(r)}(r, n)}\right)\1_{A_\infty(n)} \le 0, \qquad \qquad
 \end{aligned}
$$
 eventually, $\bbP-\mbox{a.s..}$ The last inequality  \andre{is due to  } Assumption II.
 \hfill
 \end{proof}

The next Proposition is a well known result from Brownian motion. For completeness, we include its proof in the Appendix of this paper. Denote by $H_{u}$ the hitting time of the point $u$ by the one dimensional Brownian motion.

\begin{proposition}\label{bmht} Denote by $\bbP^{x}$ the probability measure associated to  a  Brownian motion which starts from $x$. Denote by $\bbE^{x}$ the corresponding expectation.  For $0< x < a$,  we have
\begin{eqnarray}
\label{noc}  \mathbb{E}^{x} \Big[ H_{a}   \mid H_{a} < H_{0}\Big] &=& \frac 13 (a^{2} - x^{2}),\\
\label{noc2}  \mathbb{E}^{x} \Big[ H_{a} ^{2}  \mid H_{a} < H_{0}\Big] &=& \frac 1{45} \Big( 7 a^4 - 10 a^2x^2 + 3 x^4\Big)
\end{eqnarray}
\end{proposition}

\begin{remark} It is convenient, for us, to rewrite the right-hand side of equation \eqref{noc} as
\begin{equation}\label{noc3}
L_{1}(a-x, x) \Def \frac 13(a-x)^{2} + \frac 23 (a-x)x,
\end{equation}
and equation \eqref{noc2} as follows
\begin{equation}\label{noc4}
L_{2}(a-x, x) \Def \frac 1{45} \Big(7 (a-x)^4 + 28 (a-x)^3x + 32 (a-x)^2x^2 + 8 (a-x)x^3\big).
\end{equation}
\end{remark}
Notice that if both $a-x$ and $x$ are non-negative then
\begin{equation}\label{elin}
  L_{1}(a-x, x) \ge \frac 13(a-x)^{2}.
  \end{equation}
Recall that  $\vec{\phi} \andre{(k, n)}$  denotes the composition at time  $k$ in the GUP($ {\bf M}^{\ssup n}$, $M^{\ssup n}_{0} = m_{n}$,   $c_{n}f_{1}$,   $c_{n} f_{2}$) embedded into Brownian motion. Notice that \lq time\rq $\;k$ for the GUP($ {\bf M}^{\ssup n}$, $M^{\ssup n}_{n} = m_{n}$,  $c_{n}f_{1}(\cdot, n)$,   $c_{n} f_{2}(\cdot, n)$)  translates into time $T_{k}(n)$ for the Brownian motion.

%Let $\Dcal_{n}$ be the support of $\Ccal_{n}(m)$.
\begin{definition}\label{dfab}  Let
$ e(s) = e(s,n) \in \{1,2\}$  be the index $i$ such that $\phi_{i}(s+1, n) > \phi_{i}(s, n)$. Let $v(s) = v(s,n) = \{1,2\} \setminus e(s,n).$
We define $ u_i(s) = u_i(s,n)$ to be the  $s$-th element of ${\rm UP}(i,n)$, which was defined in \eqref{updef}.
More precisely, let   $u_i(1) = u_i(1,n) \Def \inf\{ k \ge 0 \colon \phi_i(k+1, n) > \phi_i(k, n)\}$, and \andre{ let
$$ u_i(s) = u(s,n)\Def \inf\{ k  > u_i(s-1, n) \colon \phi_i(k+1, n) > \phi_i(k, n)\}.$$
Define $\Ccal_{n} = \sigma\{S_{n-1},  M^{\ssup {n}}_{k} \colon  k \ge 0 \}$. }
\end{definition}

\begin{lemma}\label{indoft} Denote by $\bbE$ the expected value associated to the measure $\bbP$. For $p \in \{1, 2\}$, we have that for all $s \in \N$,
\begin{equation}
 \begin{aligned}
 \label{ftpas} \bbE[(T_{s+1}(n) - &T_{s}(n))^{p} \; |\;\Ccal_{n}] = L_{p}\left(\frac{1}{c_{n}f_{e(s)}(s, n)}, \frac{1}{c_{n}f_{v(s)}(s, n)}\right).
 \end{aligned}
 \end{equation}
\end{lemma}

\begin{proof}
Given  $(M^{\ssup n}_{s}, M^{\ssup n}_{s+1})$ the variable $(W_{T_{s+1}(n)} - W_{T_{s}(n)}, T_{s+1}(n) - T_{s}(n))$ is independent of
\begin{itemize}
\item[i)] $\Fcal_{T_s(n)},$ where recall that  $(\Fcal_t)$ is the  natural filtration of Brownian motion $\mathbf{W}$;
\item [ii)]   $W_{T_{s+1}(n)+t} - W_{T_{s+1}(n)}$, with $t \ge 0$, \andre{as  this process is a Brownian motion independent of $\Fcal_{T_{s+1}(n)}$;}
\item [iii)] $(M_{k}^{\ssup n}\colon k \in \N,  k \notin\{s, s+1\})$.  \andre{This is a consequence of points i) and ii) above.}
\end{itemize}
This together with \eqref{noc}, \eqref{noc2}  and the structure of our embedding scheme proves  \eqref{ftpas}.
\hfill
\end{proof}

\begin{remark}\label{onest}  Using \eqref{elin},
we have
 \begin{equation}\label{bain}
     L_{1}\left(\frac{1}{c_{n}f_{e(s)}(s, n)}, \frac{1}{c_{n}f_{v(s)}(s, n)}\right) \ge \frac 13 \frac{1}{c_{n}^{2}f_{e(s)}^{2}(s, n)}
   \end{equation}
   \end{remark}
\begin{lemma}\label{imin}  We have
$$
\mathbb{E}[S_{n}-S_{n-1}\;|\; \Ccal_{n}] \ge \frac 13 c_{n}^{-2}  \left(\langle \Tcal_1(m_n) \rangle \vee \langle \Tcal_2(m_n) \rangle\right)
$$
\end{lemma}
\begin{proof} Using  \eqref{noc}, the fact that
$$ S_{n}-S_{n-1} = \sum_{s=1}^\infty \big(T_{s}(n) -T_{s-1}(n)\big),$$
  and lemma \ref{indoft}, we have
\begin{equation}\label{frka1}
\begin{aligned}
\mathbb{E}[S_{n}-S_{n-1}\;|\; \Ccal_{n}]
&=  \sum_{s=1}^{\infty} \mathbb{E}[T_{s}(n)-T_{s-1}(n)\;|\;\Ccal_n]\ge  \frac 13 c_{n}^{-2} \langle \Tcal_1(m_n) \rangle +  \frac 13 c_{n}^{-2} \langle  \Tcal_2(m_n) \rangle, \; \; \bbP-\mbox{a.s.},
\end{aligned}
\end{equation}
where in the last inequality we used \eqref{bain} combined with
$$ \sum_{s=1}^\infty \frac{1}{f_{e(s)}^{2}(s, n)} =
\langle \Tcal_1(m_n) \rangle +  \langle  \Tcal_2(m_n) \rangle.$$
\hfill
\end{proof}

\begin{definition}
Let
$ N_2 = \inf\{ k \in \N \colon A_{\infty}(n) \mbox{ holds for all $n \ge k$} \}.$  By Assumption II i), combined with Borel-Cantelli Lemma, $N_1< \infty$, a.s.. Define $N = N_1 \vee N_2,$ where  $N_1$ was defined in  the statement of Proposition~\ref{sign}.
By taking a subsequence of $(m_n)$ we may and do  assume that $\bbP(N \ge n)< {\rm e}^{-n}.$
\end{definition}

%\begin{remark}\label{rimp2}
% Notice in virtue of proposition \ref{rimp}, for $n \ge N_1$, we have $W_{S_{n}} - W_{S_{n-1}} \le 0,$ a.s..
%\end{remark}
\begin{proposition}\label{bbb}
If Assumptions I and II hold then
$ \bbE[S_\infty] = \infty.$
\end{proposition}
\begin{proof}
Notice that Assumption II can be rewritten as
$$ \liminf_{n \ti} c_n^{-2} \left(\langle \Tcal_1(m_n) \rangle \vee \langle \Tcal_2(m_n) \rangle\right) >0, \qquad \bbP-\mbox{a.s.}$$
This implies that no matter which subsequence of the $m_n$ we are considering, we have
$$ \sum_{n=1}^\infty  c_n^{-2}\left(\langle \Tcal_1(m_n) \rangle \vee \langle \Tcal_2(m_n) \rangle\right)  = \infty, \qquad \bbP-\mbox{a.s..}$$
 This implies, using Lemma~\ref{imin},
$$
\begin{aligned}
\bbE[S_{\infty}] \;&\andre{\ge} \; \frac 13 \bbE\left[\sum_{n=1}^{\infty} c_{n}^{-2} \langle \Tcal_1(m_n) \rangle +  c_{n}^{-2} \langle  \Tcal_2(m_n) \rangle\right] = \infty.
\end{aligned}
$$
%For  all $n \ge N< \infty$ the event  $A_{\infty}(\ft_n, \gt_n)$ holds.  Hence the proposition is a by-product of corollary \ref{imin} and the fact that different embedding are independent.
\hfill
\end{proof}

\begin{lemma}\label{inpri} We have
\begin{equation}\label{inpri1}
  c_n^{-2} \left(\langle \Tcal_1(m_n) \rangle \vee \langle \Tcal_2(m_n) \rangle\right) \le 1, \qquad \bbP-\mbox{a.s..}
  \end{equation}
\end{lemma}
\begin{proof}
Recall the definition of $\eta_n$ given in Assumption II and $c_n$ from \eqref{cienne}. In order to prove \eqref{inpri1} it is enough to prove
$$ \eta_n \langle \Tcal_i (m_n) \rangle \le \langle B_i(m_n) \rangle.$$
The latter is proved once we prove that
\begin{equation}\label{frekt}
  \eta_n \langle \Tcal_i (m_n) \rangle = \eta_n\sum_{t \in \N} \frac 1{f^2_i(u_i(t), n)} \le \sum_{t \in \N} \frac 1{f_i(u_i(t), n)}\le \langle B_i(m_n) \rangle,
  \end{equation}
where the random sequence $u_i(t) = u_i(t, n)$ was introduced in definition~\ref{dfab}.
The first inequality in \eqref{frekt} comes from the  fact  that $\bbP(f_i(k, n) \ge \eta_n) = 1$, for all $k\ge n$, due to the definition of $\eta_n$.
The last inequality is a consequence of the definition of  $ \langle B_i(m_n) \rangle$, because the $t$-th term in $\langle B_i(m_n) \rangle$ is larger than $1/f_i(u_i(t), n).$
We used the convention $f_i(\infty, n) = \infty$, and $1/\infty =0$.
\hfill
\end{proof}

\begin{definition} Recall $e(s)= e(s, n)$ and $v(s) = v(s, n)$ from definition \ref{dfab}. For $\gamma, \theta \in \N$, set
$$\Theta_{n}(\gamma, \theta) \Def c_{n}^{-(\gamma + \theta)} \sum_{s \in \N}  \frac 1{f_{e(s)}^{\gamma}(s, n) f^{\theta}_{v(s)}(s, n)}.$$
\end{definition}
\begin{lemma}\label{faml}
$
\Theta_{n}(1, 1) \le 1,$ a.s..
\end{lemma}
\begin{proof}
%For $i \in \{1, 2\}$, $\gamma, \theta \in \N$,  define
%\begin{eqnarray}
%\Ocal_{i}(\gamma, \theta, n) = (f_{1}^{\gamma}(u_i(j), n)f_{2}^{\theta}(u_i(j), n) \colon  j \in \N)
%\end{eqnarray}
%FIrst we prove that
%$$ \langle \Ocal_{i}(1,1, n)  \rangle \le \eta_{n}^{-1} \langle B_{i}(m_n) \rangle.$$
%In fact,
As exactly one of the two quantities   $\phi_1(j+1, n) - \phi_1(j, n)$ and  $\phi_2(j+1, n) - \phi_2(j, n)$   equals one while the other is zero, we have
$$
\begin{aligned}
  \sum_{s \in \N} \frac 1{f_{e(s)}(s, n) f_{v(s)}(s, n)}  = \sum_{s \in \N} \frac 1{f_{1}(s, n) f_2(s, n)}
     &\le 2 \eta_{n}^{-1} \max_{i} \langle B_{i}(m_n) \rangle = c_{n}^{2},
 \end{aligned}
  $$
where we used again the fact  that $\bbP(f_i(k, n) \ge \eta_n) = 1$, for all $k\ge n$.
\hfill
\end{proof}
 \begin{lemma}\label{xi4} Let $s \in \{1, 2, 3, 4\}$.  We have
\begin{equation}
\Theta_{n}(s, 4-s) \le 1, \qquad \mbox{a.s..}
\end{equation}
\end{lemma}
 \begin{proof}
Notice that for any positive function $h \colon \N \mapsto (0, \infty)$, we have  that for any $p>1$,
\begin{equation}\label{trivp}
\sum_{j}^{\infty} h(j)^{-p} \le \left(\sum_{j}^{\infty} \frac 1{h(j)}\right)^{p}.
\end{equation}
For $s \in \{2,3,4\}$, using $f_{e(j)}(j, n) \wedge f_{v(j)}(j, n) \ge \eta_n$, for all $j \in \N$,  we have
 \begin{equation}\label{cas2}
 \sum_{ j \in \N}  \frac 1{f_{e(j)}^{s}(j, n) f_{v(j)}^{4-s}(j, n)} \le \eta^{-2}_n \sum_{ j \in \N}  \frac 1{f_{e(j)}^{2}(j, n)} \le \eta^{-2}_n (\langle B_{1}(m_n) \rangle + \langle B_{2}(m_n) \rangle)^2 \le c_n^4,
 \end{equation}
 where in the second inequality we used \eqref{trivp} with $p=2$ and the definition of $\langle B_{i}(m_n) \rangle$.
 \eqref{cas2} proves the Lemma for the case $s \ge 2$.
 The remaining case, i.e. $s=1$ is dealt as follows.
$$
  \begin{aligned}
 \sum_{ j \in \N}  \frac 1{f_{e(j)}(j, n) f_{v(j)}^{3}(j, n)}  &\le  \eta_{n}^{-3}
 \sum_{ j \in \N}  \frac 1{f_{e(j)}(j, n)} \le \eta_{n}^{-2} (\langle B_{1}(m_n) \rangle + \langle B_{2}(m_n) \rangle)^2 \le c_n^4,
\end{aligned}
$$
where in the last step we used both
$$  \sum_{ j \in \N} \frac 1{f_{e(j)}(j, n)} \le  \Big(\langle  B_{1}(m_n) \rangle + \langle B_{2}(m_n)\rangle\Big) , $$ and
 $
 \eta_{n}^{-1} \le (\langle  B_{1}(m_n) \rangle + \langle B_{2}(m_n) \rangle).$
 \hfill
 \end{proof}

\begin{proposition}\label{duration} There exists a constant $C$ such that  for all $ n \ge 1$, we have
\begin{eqnarray}
\label{insn3}\mathbb{E}[(S_{n}-S_{n-1})^{2}\;|\; \Ccal_{n}] \le C, \qquad \mbox{a.s..}
\end{eqnarray}
\end{proposition}
\begin{proof}
It is a consequence of \eqref{noc} and the following reasoning.    Using \andre{conditional } independence \andre{(see proof of Lemma~\ref{indoft})},   we have that for $s \neq k$,
\begin{equation}\label{ftpas1}
 \begin{aligned}
& \bbE[(T_{s+1}(n) - T_{s}(n))(T_{k+1}(n) - T_k(n)) \; |\;  \Ccal_{n}]\\
 & = \bbE\left[T_{s+1}(n) - T_{s}(n) \;|\; \Ccal_n\right] \bbE\left[T_{k+1}(n) - T_{k}(n) \;|\; \Ccal_n \right].
 \end{aligned}
 \end{equation}
We have
$$
\begin{aligned}
 \mathbb{E}&[(S_{n}-S_{n-1})^{2}\;|\; \Ccal_{n}]  =  \sum_{i=1}^\infty \sum_{k=1}^{\infty}\bbE[(T_{i+1}(n) - T_{i}(n))(T_{k+1}(n) - T_{k}(n)) \; |\; \andre{\Ccal_n}]\\
&=  \andre{\sum_{i=1}^\infty \sum_{k\neq i } \bbE[(T_{i+1}(n) - T_{i}(n))\; |\; \Ccal_n] \bbE\left[T_{k+1}(n) - T_{k}(n) \;|\; \Ccal_n \right]  + \sum_{i=1}^\infty  \bbE[(T_{i+1}(n) - T_{i}(n))^2\; |\; \Ccal_n]} \\
 &=  \andre{I + II.}
 \end{aligned}
 $$
 \andre{First we prove that I is $\bbP$-a.s. bounded by a constant.} Using \eqref{ftpas1},  \eqref{ftpas}, \eqref{noc}, \eqref{noc4}, \andre{and } the non-negativity of $f_i(k,n)$, we have
 $$
\begin{aligned}\label{inimp}
& \sum_{s=1}^\infty \sum_{k \neq i}\bbE[(T_{s+1}(n) - T_{s}(n))(T_{k+1}(n) - T_{k}(n)) \; |\;\Ccal_n]  \\
& = \sum_{s=1}^\infty \bbE\left[(T_{s+1}(n) - T_{s}(n))\;|\;  \Ccal_n\right] \sum_{k \neq s} \bbE\left[T_{k+1}(n) - T_{k}(n) \;|\; \Ccal_n \right]\\
&\le \left(\bbE[S_n - S_{n-1}\;|\; \Ccal_n]\right)^2 = \left(\frac 1{3 c^2_{n}} \langle \Tcal_1(m_n) \rangle  + \frac 1{3 c^2_{n}} \langle \Tcal_2(m_n) \rangle
 + \frac 23 \Theta_{n}(1,1) \right)^{2}, \qquad \bbP-\mbox{a.s..}
\end{aligned}
$$
In virtue of Lemmas \ref{inpri} and \ref{faml}, each of the terms appearing in the previous equation are bounded by a \andre{constant}  which is independent of $n$.  \andre{Next we prove that II is a.s. bounded by a constant.}
%Define
% $$ \Mcal_{i}(n) =  \{f^{4}_{i}(k) \colon f_{i}(k+1) - f_{i}(k)>0\}, $$
%
Notice that from corollary \ref{onest} with $p=2$, we have
$$ \sum_{s=1}^\infty   L_{2}\left(\frac{1}{c_{n}f_{e(s)}(s, n)}, \frac{1}{c_{n}f_{v(s)}(s, n)}\right) =  \frac 7{45 }\Theta_n(4, 0)  +\sum_{s=1}^{3} h_{s} \Theta_{n}(s, 4-s),$$
where $h_1 = 28/45$, $h_2 = 32/45$, and  $h_3 = 8/45$.
Hence,  using corollary~\ref{onest}, we have
\begin{equation}\label{conmyf}
\begin{aligned}
&\sum_{s=1}^{\infty} \bbE[(T_{s+1}(n) - T_{s}(n))^{2}\;|\; \Ccal_n ] = \frac 7{45 }\Theta_n(4, 0)  +\sum_{s=1}^{3} h_{s}\Theta_{n}(s, 4-s),
\end{aligned}
 \end{equation}
In virtue of Lemma~\ref{xi4},  each of the terms appearing in the previous equation are bounded by a constant which is independent of $n$, ending the proof.
\hfill
\end{proof}

\begin{lemma}\label{doobc}
$\bbE[\sup_{n \le N}  W^2_{S_{n}}] < \infty.$
\end{lemma}
\begin{proof}
Using Jensen's inequality,  and setting $S_0 = 0$, we have
$$
\bbE[S_{j}^{2}] = \bbE\left[\left(\sum_{\ell =1}^j (S_\ell - S_{\ell -1})\right)^2\right]  \le j \sum_{i=1}^j \bbE[(S_{{i}} - S_{{i-1}})^{2}] = O( j^2),$$
where in the last step we used Lemma~\ref{xi4}.
Notice that using Burkholder-Davis-Gundy (BDG) inequality, we have
$$  \bbE[\sup_{\ell \le j} W^4_{S_{\ell}}] \le C \bbE[S_{j}^{2}]=  O(j^{2}) ,$$
for some constant $C$. Hence
$$
\begin{aligned}
\bbE[\sup_{n \le N}  W^2_{S_{n}}] &\le   \sum_{j=1}^{\infty}  \bbE[ \sup_{\ell \le j} W^2_{S_{\ell}} \1_{\{N = j\}}]\le   \sum_{j=1}^{\infty} E[\sup_{\ell \le j} W^4_{S_{\ell}}]^{1/2} P(N = j)^{1/2}\\
&\le C \sum_{j=1}^{\infty} \left( \bbE[S_{j}^{2}] \right)^{1/2} P(N = j)^{1/2} \le  C'  \sum_{j=1}^{\infty} j {\rm e}^{- j/2} <\infty.
\end{aligned}
$$ \hfill
\end{proof}

Define $\overline{W}_{\infty} = \liminf_{n \ti} W_{S_{n}}$. The random variable $\overline{W}_{\infty}$ can be seen as an infinite sum of independent random variables $W_{S_n} - W_{S_{n-1}}$. Moreover,  in virtue of Proposition~\ref{sign}, we have that $W_{S_{n}}$ is non-increasing for $n \ge N$. This implies that $\overline{W}_{\infty} = \lim_{n \ti} W_{S_{n}}$. Hence, by Kolmogorov 0-1 law,   $\overline{W}_{\infty}$ is either a.s.  finite or it equals $- \infty$  a.s..
We distinguish  the two cases.\\
\underline{Case I.} In this case, we assume that
\begin{equation}\label{finsumod}
 \overline{W}_{\infty} >  - \infty, \qquad \bbP-\mbox{a.s.}.
\end{equation}
\begin{lemma}\label{ifw} If we assume \eqref{finsumod} then we can choose a subsequence $m_{n}$, with $n\ge 1$,  such that $\bbE[\overline{W}_{\infty}^{2}] <\infty.$
\end{lemma}
\begin{proof}      As $\overline{W}_{\infty}$ can be represented \andre{as}  the sum of independent variables $Y_{n} = W_{S_{n+1}} - W_{S_{n}}$ we have, using the Kolmogorov Three Series Theorem (see for example page 115 of \cite{W90}),  that
\begin{equation}\label{secmom}
 \sum_{n=1}^{\infty} {\rm Var}(Y_{n} \1_{|Y_{n}| \le  1} ) < \infty.
 \end{equation}
On the other hand
\begin{equation}\label{varec}
\begin{aligned}
 {\rm Var} (Y_{n} \1_{|Y_{n}| > 1}) &\le  \bbE[Y^{2}_{n} \1_{|Y_{n}| > 1}] \le \bbE[Y_{n} ^{4}]^{1/2}\bbP(|Y_{n}| > 1)^{1/2}\\
 &\le C \bbE[(S_{n+1} - S_{n})^{2}]^{1/2} \bbP(|Y_{n}| > 1)^{1/2} \le C_{1} \bbP(|Y_{n}| > 1)^{1/2}.
\end{aligned}
\end{equation}
The inequality before the last one is an application of  BDG inequality, and the last inequality is a consequence of  Proposition~\ref{duration}.
By the three-series Theorem, we have that
$$ \sum_{n} \bbP(|Y_{n}| > 1) <\infty.$$
The latter implies that $\lim_n \bbP(|Y_{n}| > 1) =0$. Hence we can choose a subsequence  \andre{$m'_n$ of $m_n$ } such that
\begin{equation}\label{varec1} \sum_{n} \bbP(|Y_{n}| > 1)^{1/2} <\infty.
\end{equation}
\andre{Notice that the embedding times $S_n$ obtained using $m'_n$ satisfy \eqref{finsumod}. Next we  work  with the embedding obtained using the sequence  $m'_n$}.  Combining \eqref{varec1} with \eqref{varec}, we have that
 \begin{equation}\label{varec2}
 \sum_{n=1}^{\infty}{\rm Var} (Y_{n} \1_{|Y_{n}| > 1})<\infty.
 \end{equation}
 Using the trivial inequality $(a+b)^{2} \le 2 a^{2} + 2 b^{2}$, we have
 $$
 \begin{aligned}
   {\rm Var} (Y_{n}) &\le 2 \bbE\left[\Big(Y_{n}\1_{|Y_{n}| \le  1} - \bbE[Y_{n}\1_{|Y_{n}| \le  1}]\Big)^{2}\right] +
 2 \bbE\left[\Big(Y_{n}\1_{|Y_{n}| >  1} - \bbE[Y_{n}\1_{|Y_{n}| >  1}]\Big)^{2}\right]\\
 &= 2{\rm Var} (Y_{n} \1_{|Y_{n}| \le  1}) + 2 {\rm Var} (Y_{n} \1_{|Y_{n}| >  1})
 \end{aligned}
 $$
Combining this with \eqref{secmom} and \eqref{varec2}, we get
$$ \sum_{n=1}^{\infty} {\rm Var} (Y_{n}) \le 2 \sum_{n =1}^{\infty}{\rm Var} (Y_{n} \1_{|Y_{n}| \le  1}) + 2 \sum_{n =1}^{\infty}{\rm Var} (Y_{n} \1_{|Y_{n}| >  1}) < \infty.$$
 Recall that $Y_{n}$ is a zero-mean random variable. Hence ${\rm Var} (Y_{n}) = \bbE[Y^{2}_{n}]$. Using Fatou's Lemma,
we have
$$ \bbE[\overline{W}_{\infty}^{2}] \le \liminf_{n \ti} \bbE[W_{S_{n}}^{2}] =  \lim_{n \ti} \bbE[W_{S_{n}}^{2}]  = \sum_{i=1}^{\infty} {\rm Var} (Y_{i})<\infty.$$
This implies our result.
  \hfill
\end{proof}

Recall that from Lemma~\ref{doobc} we have $\bbE[\sup_{n \le N} W^{2}_{S_{n}}]<\infty.$
Notice that for any sequence $a_{n} \andre{\in \R}$ which has the property that there exists
an integer $M$ such that $a_{n+1} \le a_{n}$ for all $n\ge M$, we have
\begin{equation}\label{Jobj}
 |a_{n}| \le (\sup_{j \le M} |a_{j}|)  + \andre{a_M - a_{\infty}},
\end{equation}
where $ a_{\infty} = \lim_{n\ti} a_{n}$. In fact, the inequality is trivial for $n \le M$. For $n>M$ notice that the sequence \andre{$|a_M - a_n| = a_M - a_n$} is non-decreasing in $n$. Hence $ |a_M - a_n| \le a_M -  a_{\infty}$. Hence, for $n >M$ we have
$$ |a_{n}|  \le |a_M| + |a_n - a_M| \le |a_M| + \andre{(a_M -a_{\infty})} \le (\sup_{j \le M} |a_{j}|)  + \andre{(a_M -a_{\infty})},$$
proving \eqref{Jobj}.
\andre{Using} $(a+b)^2 \le 2 a^2 + 2 b^2$, we have
\begin{equation}\label{JS}
  a_{n}^{2} \le 2  (\sup_{j \le M} |a_{j}|)^{2} + 2(a_{\infty} - a_{M})^{2}.
  \end{equation}

\andre{By} the Optional Sampling Theorem, we have that $\bbE[W^2_{S_n}] = \bbE[S_n]$, for all $n \ge 1$. Hence,    for $n \ge 1$,
	\begin{equation}\label{contr2}
	\begin{aligned}
	  \mathbb{E}[S_{n}] &=  \mathbb{E}[ W^2_{S_{n}}] \le    2\bbE [\sup_{i \le N}W_{S_{i}}^2]   + 2  \bbE[(\overline{W}_{\infty} - W_{S_{N}})^{2}] \qquad \mbox{(using \eqref{JS})}\\
	 &\le 2 \bbE [\sup_{i \le N}W_{S_{i}}^2] +4 \bbE[\overline{W}_{\infty}^{2}] + 4  \bbE[W_{S_{N}}^{2}]<\infty,
	\end{aligned}
	\end{equation}
	\andre{where we used Lemmas ~\ref{doobc} and \ref{ifw} to establish that the latter quantity is finite.}
	Equation \eqref{contr2} yields, by sending $n \ti$ and using monotone convergence Theorem, that   $\bbE[S_\infty] < \infty$ which contradicts Proposition~\ref{bbb}.\\
\underline{Case II}.
Assume  that
$\overline{W}_{\infty}= - \infty,$  a.s..
The process $Z_{n} = W_{S_{n}} $, with $n\ge 0$, is a martingale. In this case we have that
\begin{equation}\label{inm}
\lim_{n\ti} Z_{n} = - \infty.
\end{equation}
For $n > N$ we have $Z_{n} - Z_{n-1} \le 0,  \mbox{a.s..}$
This implies   that for all  $n \ge 1$, we have
\begin{equation}\label{fre2}
\andre{ \bbE[\sup_{n} Z_{n}^{+}] } = \bbE\left[\sup_{n \le N} \Big(Z_{n}^{+}\Big)\right]  \le  1+ \bbE\left[\sup_{n \le N} \Big(Z_{n}^{+}\Big)^2\right] = 1 + E[\sup_{n \le N} (W_{S_n})^2] <\infty,
\end{equation}
where in the last step we used Lemma~\ref{doobc}.
Notice that, as $Z_{n}$ is a zero mean martingale, we have
$$\sup_{n} \bbE[|Z_{n}|] = 2 \sup_{n} \bbE[Z_{n}^{+}] < \infty.$$
Hence, the martingale convergence theorem implies that $\lim_{n \ti} Z_{n} = Z_{\infty}$ exists and is integrable. This contradicts \eqref{inm}.
\subsection{Proof of Theorem \ref{rubin} in the case $\langle B_{1}(m_n)\rangle \vee \langle B_{2}(m_n)\rangle = \infty$, eventually.}
\andre{Recall that w}e are assuming that both Assumption I and II hold and again reason by contradiction. In this case, to complete the proof of Theorem \ref{rubin} we suppose that  $\langle B_{1}(m_n)\rangle \vee \langle B_{2}(m_n)\rangle = \infty$ for all  but finitely many $n \in \N$,  where $m_n$ is the sequence described in Assumption II.
By taking a subsequence, we assume that $\langle B_{1}(m_n)\rangle \vee \langle B_{2}(m_n)\rangle = \infty$ for all   $n \in \N$.
  Hence, in virtue of Assumption 1, there exists an index  $j \in\{1,2\}$ such that $\langle B_{j}(m_n)\rangle = \infty$  for all  $n \in \N$. Without loss of generality, we suppose that $\langle B_{2}(m_1)\rangle = \infty$  for all $n \in \N$ while \andre{$\langle B_{1}(m_n)\rangle < \infty$}  for all $n \in \N$.  In order for Assumption II to hold, we have
\begin{equation}\label{in}
\langle \Tcal_2(m_n)\rangle = \infty, \qquad \bbP-\mbox{a.s.},
\end{equation}
for all  $n \in \N$. Fix $n$ such that \eqref{in} holds and $\langle B_{2}(m_n)\rangle = \infty$.
This implies that
$$ \sum_{k=1}^\infty \frac 1{f_2(u_2(k), n)} = \infty, \qquad \bbP-\mbox{a.s.},$$
\andre{where the  sequence $u_i(t) = u_i(t, n)$ was introduced in definition~\ref{dfab}.}
  We embed the GUP($\mathbf{M}, M_0 = m_n, f_1(\cdot, n), f_2(\cdot, n))$ into Brownian motion starting at 0, in the way we described in Section 2.3. Notice that we embed just one GUP.  In this embedding we choose  $a(k) = 1$ and $b(k) =2$ for all $k \in \N$. In this case, on $A_\infty$, the Brownian motion would hit $-\infty$ before hitting $\langle B_1(m_n) \rangle <\infty$, which gives a contradiction, as the Brownian motion is recurrent in one dimension.

\section{Proof of Theorem~\ref{cyclethm}}

% For any pair of  sequences $\{a_{n}\}, \{b_{n}\}$, we say $a_{n} \sim b_{n}$ if
% $$0 < \liminf_{n \ti }a_{n}/ b_{n} \le \limsup_{n \ti }a_{n}/ b_{n} < \infty.$$

From now on we fix $\ell \in [K]$ and assume that the assumptions of Theorem~\ref{cyclethm} hold.
Recall  the  Markov process
$ M_k = (\vec{X}_k, \vec{N}_k, \vec{T}_k)$, with  $ k \in \N_0$,   which takes values in the space $\S $ defined in Section~\ref{appl}.

 The next result establishes that the jumps from each vertex $j \in V$ of the $\ell$-process can be modelled using a  suitable GUP.
In order to simplify the notation, as  $\ell \in [K] =  \{1, 2, \ldots, K\}$ is fixed, we remove it from most of the notation used  throughout this section, with the exception of the process ${\bf X}^{\ssup \ell}$ itself.

\begin{lemma}\label{speGUP}
 Fix $j \in V$.  Consider $\Sc(m)$.
%  Denote by $a_j$ the $j$-th coordinate of $a$. Consider an urn starting with $a_{j \ominus 1}$ white and $a_{j \oplus 1}$ red balls.
  Suppose that each time the process ${\bf X}^{\ssup \ell}$  jumps from $j$ to $j \ominus 1$ (resp.  to $j \oplus 1$)  we add one white ball (resp. one red ball) to a given urn.  This urn evolves like a GUP(${\bf \widetilde{M}}, \widetilde{M}_0,  f_1, f_2$), for some choice  of reinforcement processes $f_i (\cdot)$, $i\in \{1, 2\}$ and Markov process ${\bf \widetilde{M}}$ on $\S$. This representation is true up to the random time, possibly infinite, of the last visit of the process $\mathbf{X}^{\ssup \ell}$ to vertex $j$.
\end{lemma}
\begin{proof}
Set
$t(1, j) = \inf\{ u \colon u \in \N,\, X^{\ssup \ell}_u = j \}.$
Throughout this paper, the infimum over an empty set is $+\infty$.
Define recursively, for $ k \ge 2$,
$$ t(k, j) = \inf\{u \colon u > t(k -1, j),\, X^{\ssup \ell}_u = j\}. $$
In the remaining part of the proof of this lemma, we simplify the notation into $t(k)$, removing the reference to the vertex $j$, which is fixed.
Set $\widetilde{M}_k =   M_{t(k+1)}.$
Define the  processes
\begin{eqnarray*}
 f_1(k+1) &=&   T_{\ell, t(k+1)}({j}), \qquad f_2(k+1) = T_{\ell, t(k+1)}({j \ominus 1}).
\end{eqnarray*}
Next we show that these processes satisfy the  conditions listed in the definition of GUP.  In fact, referring to the definition of GUP, given in Section \ref{GUP},  we have the following.
\begin{itemize}
\item[A)] It is satisfied. The processes $f_i(k+1)$ and $\phi_i(k+1)$ (the latter is the composition of the urn described above, at time $k+1$), with $i\in \{1,2\}$,  are  determined by  $\widetilde{M}_k$.
\item[B)] $\S$ is  countable. This is a consequence of   our assumption on the family of random variables $\Xi_{\ell,j}(s)$.
\end{itemize}
The resulting GUP models the jumps  of ${\bf X}^{\ssup \ell}$ from vertex $j$ to its neighbors, up to the last visit of this process to $j$.
\hfill
\end{proof}
%\begin{figure}
%\hspace*{2cm}\includegraphics[height=5in]{clock}
%\vspace*{-5cm}
%\caption{GUP related to the j-th vertex.}
%\end{figure}

\begin{definition}
We denote the GUP described in Lemma~\ref{speGUP} by GUP($j, m$), where $m$ is the initial configuration of $\mathbf{M}$. Define $A_\infty(j)$ to be the event that  GUP($j, m$) picks infinitely many balls of each colour. We emphasize the fact that these objects depend on $\ell$. However, in our reasoning below, we decouple the behaviour of $\mathbf{X}^{\ssup \ell}$ from the other processes.
\end{definition}
%Next, we define other type of GUP by aggregating exactly two adjacent vertices as follows.
For $m \in \S$ and $ j\in E$, $N_{\ell, j}(0)$ under the measure $\bbP_m$ is deterministic.  We denote using
$\pi_j(m) = \pi^{\ssup \ell}_j(m)$ the value of $N_{\ell, j}(0)$ under $\bbP_m$.
% The quantities $\pi_j(m)$ determine the transitions probability of $\mathbf{X}^{\ssup \ell}$ at first step, under $\bbP_m$, i.e.
%$(\pi_j(m), j \in E)$ is the sequence of initial transition weights for this process.
Define the event
$$ I  = I^{\ssup \ell} \Def \Big\{ \exists V' \subset V \colon X^{\ssup \ell}_n\in V'\mbox{ for all large $n$}\Big\}.$$
%$$ D(k, a) = D(\ell, k, a) = \sum_{j = \pi_{k}(a)}^\infty \frac 1{j^\alpha}.$$
%For each $a \in \S$, there exists $\iota(a) \in V$,  such that
Notice that $V' \subset V$ implies that $V' \neq V$.
As $\mathbf{X}^{\ssup \ell}$ jumps to nearest neighbors, we can assume that $V'$ is a connected set.
Our aim is to prove that $I$ holds with probability 1, no matter what is the initial configuration in $\S$. More precisely we prove that $V'$ can be taken as the set  of  two adjacent vertices. Of course this pair of vertices is random.
%$$ D(k, a) = D(\ell, k, a) = \sum_{j = \pi_{k}(a)}^\infty \frac 1{j^\alpha}.$$
%For each $a \in \S$, there exists $\iota(a) \in V$,  such that

Set $R_1$ to be  the smallest   even number larger than  both $R+2$ and $v$.
Define the  set $\Pi  = \Pi^{\ssup \ell}$ as follows.
$$
\begin{aligned}
\Pi \Def\Big\{\widetilde{m}\colon  \exists j \in E  \mbox{ such that }&\pi_{j}(\widetilde{m}) + R_1  \le  \pi_{j \ominus 1}(\widetilde{m}) \wedge  \pi_{j \oplus 1}(\widetilde{m}) \mbox{ and } \pi_{j \oplus 1} (\widetilde{m})\neq \pi_{j \ominus 1}(\widetilde{m})
\Big\}.
\end{aligned}
$$
%The proof of the following result follows an easy reasoning but is somehow technical.
\begin{lemma}\label{impin} Consider $\Sc(m)$, for $m \in \S$.  We have
\begin{equation}
  \bbP_m \Big( I \cup \big\{M_t \in \Pi, \mbox{i.o.}\}\Big) =1.
  \end{equation}
\end{lemma}
\begin{proof}
Consider $\Sc(m)$.   Throughout this proof, the constant  $ "c"$ stands for a generic positive constant which depends on \andre{$R_1$} but does not depend on the initial state $\widetilde{m}$ of the Markov chain $\mathbf{M}$, with $\widetilde{m} \notin \Pi$.
% Suppose that for each $j \in E$  we have

%For any configuration $\widetilde{m} \in \S$ we say that $\widetilde{m}$ is good if it satisfies the following
%$$ \exists  \kappa \in E \colon
%  \pi_{\kappa}(\widetilde{m}) + R  \le  \pi_{\kappa \ominus 1}(\widetilde{m}) \wedge  \pi_{\kappa \oplus 1}(\widetilde{m}).$$
Next, for each configuration $\widetilde{m} \notin \Pi$,
we distinguish two main cases.
\begin{enumerate}
\item Suppose   $M_{0} =\widetilde{m}$ satisfies
\begin{equation}\label{mtil1}
|\pi_{j}(\widetilde{m})  - \pi_{j \ominus 1}(\widetilde{m} )| \le R_1, \qquad  \mbox{for all  $j \in E$}.
\end{equation}
  We prove, for this case, that
\begin{equation}\label{flowB}
  \bbP_{\widetilde{m}}\Big(M_{5R_1  +4}\in \Pi\Big) > c,
  \end{equation}
where recall that $c>0$ is a constant independent of $\widetilde{m}$.
Our strategy to prove \eqref{flowB} is to prove the existence of a path of length $5 R_1  +4$, with probability at least $c$, which makes $M_{5  R_1  +4 } \in \Pi$. We describe the path as follows. Consider the largest between $\pi_{X_0^{\ssup \ell}\ominus 1}(\widetilde{m})$ and $\pi_{X_0^{\ssup \ell} }(\widetilde{m})$. Suppose the latter is the largest. It is easy to  adapt the following argument to the other case. Define the event
$$ B_1 \Def \{X_s^{\ssup \ell} \in \{ X_0^{\ssup \ell},  X_0^{\ssup \ell} \oplus 1\} \mbox{ for all  } s \in 1, 2, \ldots,  3 R_1 +2\}.$$
Notice, that as $R_1$ is even, on the event $B_1$ we have that $X^{\ssup \ell}_{3R_1+2} = X^{\ssup \ell}_0$.
Define $B_2 = \{X^{\ssup \ell}_{3 R_1 + 3} = X^{\ssup \ell}_{0} \ominus 1\}$. Finally let
$$ B_3 \Def \{X_s^{\ssup \ell} \in \{ X_0^{\ssup \ell} \ominus 2,  X_0^{\ssup \ell} \ominus 1 \} \mbox{ for all  } s \in 3 R_1 + 3, 3 R_1 + 4, \ldots,  5  R_1 +4\}.$$
Notice that on $B_1 \cap B_2 \cap B_3$ we have that $M_{ 5 R_1  +4} \in \Pi$, and  that we use $3R_1 +2$ in the definition of $B_1$ in such a way that the  condition $\pi_{j \oplus 1} (\widetilde{m})\neq \pi_{j \ominus 1}(\widetilde{m})$, appearing in the definition of $\Pi$, is satisfied. Next, we bound from below $\bbP_{\widetilde{m}}(B_1 \cap B_2 \cap B_3)$ when $\widetilde{m}$ satisfies \eqref{mtil1}.
As we already pointed out, the event $B_1 \cap B_2 \cap B_3$ depends on the first $5 R_1 +4$ consecutive jumps of $\mathbf{X}^{\ssup \ell}$.  During the first $ 5 R_1 +4$ steps, the difference between adjacent  weights can be at most $6  R_1 +4$, which is a rough upper bound.
This is because $\widetilde{m}$ satisfies \eqref{mtil1}. Hence, each path of length $ 5 R_1 +4$ has a probability of at least
\begin{equation}\label{runoc}
\left( \frac{1}{1 + (1+6 R_1+4+R)^\alpha} \right)^{5 R_1 +4}.
  \end{equation}
 \andre{It follows from   the fact that the function }
 $$ x \mapsto \frac{x^\alpha}{x^\alpha + (x+ 6 R_1+4+R)^\alpha}$$ is monotone increasing, for $x \in [1, \infty)$.
\item  Suppose that $\widetilde{m} \notin \Pi$ satisfies
\begin{equation}\label{msat}
 \max_{j \andre{\in E}} |\pi_{j}(\widetilde{m})  - \pi_{j \ominus 1}(\widetilde{m})\vee \pi_{j \oplus 1}(\widetilde{m})| > R_1.
\end{equation}
%Next, we prove that
%\begin{equation}\label{lowbC}
%   \bbP_{\widetilde{m}}(\{\exists k \ge 0 \colon M_k \in \Pi\} \cup I) \ge c.
%\end{equation}
Notice that \eqref{msat}, together with $\widetilde{m} \notin \Pi$ and the fact the underlying graph is a polygon,  implies that
\begin{equation}\label{emptj}
 \Big\{j  \in E\colon |\pi_{j}(\widetilde{m})  - \pi_{j \oplus 1}(\widetilde{m})| \le  R_1\Big\} \neq \emptyset.
\end{equation}
\andre{To prove \eqref{emptj}, it is enough to reason by contradiction. If \eqref{emptj} does not hold, then the minimizer $j^*$  of $j \mapsto \pi_{j}(\widetilde{m})$ would satisfy  $\pi_{j^*}(\widetilde{m}) +R_1 \le \pi_{j^* \oplus 1}(\widetilde{m}) \vee \pi_{j^* \ominus 1}(\widetilde{m})$, which contradicts  $\widetilde{m} \notin \Pi$.
}
%In fact, if \eqref{emptj} does not hold, then, as the underlying graph is a cycle, there exists a $\kappa \in E$ such that
%$$\pi_{\kappa}(\widetilde{m}) + R  \le  \pi_{\kappa \ominus 1}(\widetilde{m}) \wedge  \pi_{\kappa \oplus 1}(\widetilde{m})$$
%proving that    $\widetilde{m} \in \Pi$ and yielding a contraddiction.

Consider $A^*\subset V$ to be one among the  largest  non-empty set of consecutive edges  with the following property. If $j \in A^*$, then
$$  |\pi_{j}(\widetilde{m})  - \pi_{j \oplus 1}(\widetilde{m})| \le  R_1.$$
In virtue of \eqref{msat}, we have $A^*\neq V$.
Denote by $D$ the  second largest subset of consecutive edges, disjoint from $A^*$, with the properties described above. Notice that $D$ could be an empty set. On the opposite, $D$ could have the same size as $A^*$. \\
Suppose that $A^*$ consists of exactly $h<v$ edges, with $h \ge 1$, say $A^*= \{a, a \oplus 1,  a \oplus 2, \ldots, a \oplus (h-1)\}$. To simplify the notation, we assume that $A^*= \{0, 1, 2, \ldots, h-1\}$ and leave to the reader the simple task to adapt the following reasoning to the general $A^*$.
Notice that in what follows, to simplify the reasoning, we ignore the  second condition appearing in  the definition of $\Pi$, i.e.  $\pi_{\kappa \oplus 1} (\widetilde{m})\neq \pi_{\kappa \ominus 1}(\widetilde{m})$. This would affect the path by just one step (stop either one step earlier or one step later).
  \\
 First, \underline{ suppose that  $h \ge 2$.}
We first deal with the case where the   probability  of the event
\begin{equation}\label{buu}
  \bbP_m(\{\mbox{$\mathbf{X}^{\ssup \ell}$  hits $1$ before $h\andre{-1}$}\} \cup I) \ge 1/2,
  \end{equation}
 Set $B_4 \Def \{\mbox{$\mathbf{X}^{\ssup \ell}$  hits vertex $1$ before $h-1$}\} \cup I.$
 Recall that $t_1 = t(1,1)$ is  the hitting time of vertex $1$ \andre{by the process $\mathbf{X}^{\ssup \ell}$.}
 We have
 \begin{equation}\label{stanc}
 \bbP_{\widetilde{m}}(\{{X}^{\ssup \ell}_k  \in \{1, 2\} \mbox{ for all $t_1 \le   k \le  t_1 + 2R_1+1$}\}\cap B_4) \ge \frac 12\left(\frac{1}{2 + 3R_1+1+ R}\right)^{2R_1+1},
 \end{equation}
yielding the result for this case. \andre{In fact, on the event appearing in the probability in the right-hand side of \eqref{stanc}, we have }
$$\andre{\pi_0(M_s) +R_1  \le  \pi_{m-1}(M_s)\wedge \pi_{1}(M_s),}$$
\andre{for $s \in \{t_1 + 2R_1, t_1 + 2R_1+1\}$. As for the requirement $\pi_{m-1}(M_s) \neq \pi_{1}(M_s)$, as discussed above,  it holds for at least one element of the set $\{M_{t_1 + 2R_1}, M_{t_1 + 2R_1+1}\}$.}

When \eqref{buu} does not hold, replace $\{1, 2\}$ with $\{h-2, h-1\}$ in the left-hand side of \eqref{stanc}.\\
\underline{Next, suppose that $h = 1$.}
We distinguish two further subcases. The number of edges in $D$ can be either $0$ or $1$. 
 First consider the case when $D$ has exactly one edge, which we  denote by $d$.
Assume  that under  the initial configuration $\widetilde{m}$ we have $1 \le X_0^{\ssup \ell} \le d$.
 We can adapt easily the following  argument to the other case.
Let $U = \{1, d+1\}$.  Let $s_1 = \inf\{t \ge 0 \colon X_t^{\ssup \ell}  \in U\}$. Define recursively $s_k = \inf\{t > s_{k-1} \colon X_t^{\ssup \ell}  \in U\}$.
Set $\widetilde{U} = \{2, d\}$.
Notice that
\begin{equation}
\bbP_{\widetilde{m}}(\{\mbox{$s_k <\infty$ and ${X}^{\ssup \ell}_{s_k +1} \in \widetilde{U}$  for all $k \le 4R_1$}\} \cup I) \ge  \left(\frac{1}{1 +(1+R)^\alpha}\right)^{4R_1}.
\end{equation}
 Notice that on the event inside the previous probability, we have that  either $I$ holds or $\mathbf{M}$ hits $\Pi$.\\
 \underline {Finally consider the case when $D$ is empty,}  and $A^*$ consists of exactly one edge .  Set
 $$
 \begin{aligned}
 \mathcal{X} \Def \{m \in \S \colon &\mbox{$A^* $ consists of  exactly one edge  and   $D$ is empty}\}
% \mathcal{X} \Def \bigcup_{j \in E} \{m \in \S \colon &\mbox{$A $ consists of  exactly one edge, which is  $j$,  $D$ is empty and }\\
% &\quad  \mbox{either } \pi_{j\ominus1 }(m)  \le 5 (\pi_j(m) \vee R_1)  \mbox{ or }   \pi_{j\oplus2 }(m)  \le 5 (\pi_{j \oplus 1}(m) \vee R_1) \}.
 \end{aligned}
 $$
% Next, we prove that either $I$ holds or  there is  a positive probability that the process $\bf{M}$  hits a configuration in  $\mathcal{X} $.\\ 
Again, we distinguish two cases. Define 
\begin{equation}\label{belleve}
\begin{aligned}
 \mathcal{Y}\Def \{m \colon\quad &\exists j \in E  \quad \mbox{such that } j \in A^* \mbox{ and either }  \pi_{j\ominus1 }(m)  \le K (\pi_j(m) \vee R_1)  \mbox{ or }   \\
& \pi_{j\oplus2 }(m)  \le K (\pi_{j \oplus 1}(m) \vee R_1), \mbox{ or both}\},
\end{aligned}
 \end{equation}
where the constant $K$ is chosen to be large enough, as specified later.
Suppose that $m \in \mathcal{X}\cup \mathcal{Y}$.
 Suppose that $  \pi_{j\ominus1 }(m)  \le K (\pi_j(m) \vee R_1)$. Then,
 \begin{equation}\label{misabb}
 \bbP_m\Big(I \cup \Big\{\{t_j<\infty\} \cap \big\{X_k \in  \{j, j \oplus 1\} \mbox{ for all } k \in [t_j,  t_j +2R_1+1]\big\}\Big\}\Big) \ge  \left(\frac 1{1+ (K+2R_1+R)^\alpha}\right)^{2R_1+1}.
 \end{equation}
 If the  event  appearing inside the probability in  \eqref{misabb} holds, either $I$ holds or the process $\mathbf{M}$ hits $\Pi$. A similar result holds when $\pi_{j\oplus2 }(m)  \le K (\pi_{j \oplus 1}(m) \vee R_1)$.\\
 Next, we consider the case  when the initial state $m \in \mathcal{X} \cap \mathcal{Y}^c$. We show that either $I$ holds, or there exists an a.s. finite time $t$, under $\bbP_m$,  such that  $M_t \in \mathcal{X}^c \cup \mathcal{Y}$. It means that either $I$ holds or we can use results from other cases to infer that $\mathbf{M}$ hits $\Pi$.  We reason by contradiction. For each $j\in E$, define the following event 
$$ \mathcal{A}(j) = \bigcap_{k=0}^\infty\{N_{j\ominus 1}(k)> K(N_j(k)\vee R_1)\}\cap \{ N_{j \oplus 2}(k)> K(N_{j \oplus 1}(k) \vee R_1)\} \cap \{|N_j(k) -N_{j \oplus 1}(k)| \le R_1\}.$$ 
 We have that $\{M_k \in   \mathcal{X} \cap \mathcal{Y}^c, \forall k \in \N_0\} \subset \bigcup_{j \in E} \mathcal{A}(j).$
 Suppose $\bbP_m(\mathcal{A}(j))>0$ for some $j \in E$.   Consider the following urn.
   Suppose that each time that ${\bf X}^{\ssup \ell}$  jumps from vertex $j$ to $j\ominus1$ or from  $j \oplus 2$ to $j \oplus 3$  we add one \andre{red} ball  and each time it jumps either from vertex 
$j$ to $j \oplus 1$ or from $j \oplus 2 $ to $j \oplus 1$ we add one \andre{white} ball. \andre{We show that } this urn evolves like a
GUP(${\bf M'}, M'_0 = m,  f_1, f_2$), for some choice  of reinforcement processes $f_i (\cdot)$, $i\in \{1, 2\}$ and Markov process ${\bf M'}$ on $\S$. To see this, define
$$ s_1 \Def \inf\{u \colon u \ge 0, \mbox{either } X^{\ssup \ell}_u = j \mbox{ or } X^{\ssup \ell}_u = j \oplus 2\}.$$
Define recursively, for $k \ge 2$,
$$ s_k \Def \inf\{u \colon u > s_{k -1}, \mbox{either } X^{\ssup \ell}_u = j \mbox{ or } X^{\ssup \ell}_u =  j \oplus 2\}. $$
 Define  the stochastic processes
\begin{eqnarray*}
 f_1(k+1) &\Def& T_{j} (s_k)\1_{X^{\ssup \ell}_{s_k} =j} + T_{j \oplus 1} (s_k)\1_{X^{\ssup \ell}_{s_k} = j \oplus 2},\\
 f_2(k+1) &\Def& T_{j \ominus 1} (s_k)\1_{X^{\ssup \ell}_{s_k} = j} + T_{j \oplus 2} (s_k)\1_{X^{\ssup \ell}_{s_k} = j \oplus 2}.
\end{eqnarray*}
where $T_j = T_{\ell, j}$.
Let $M'_{k} \Def M_{s_{k+1}}$, for $k\in \N_0$. It is trivial to check that A) and B) from the definition of GUP are both satisfied.  For $\mathcal{A}(j)$ to hold, either $I$ must hold or infinitely many balls of each color must be picked.
Suppose that $\bP_m(\mathcal{A}(j) \cap A_\infty)>0$, where, as before, $A_\infty$ is the event that infinitely many balls of each color are picked in this urn. In virtue of the proof of Proposition~\ref{prap}, we can find a sequence $m^*_n$ of configurations in $\S$, such that 
\begin{equation}\label{condmsta}
\begin{aligned}
 \{ \pi_{j \ominus 1}(m^*_n)> K (\pi_j(m^*_n) \vee R_1)\}&\cap \{\pi_{ j \oplus 2}(m^*_n)> K (\pi_{j \oplus 1}(m^*_n)\vee R_1)\}\\
&\cap \{|\pi_j(m^*_n) - \pi_{j \oplus 1}(m^*_n)| \le R_1\},
\end{aligned}
\end{equation}
and $\lim_n\bP_{m^*_n}(A_\infty)  =1$.  Next we prove that this GUP would satisfy both Assumption I and II, yielding a contraddiction.
 As for Assumption I), it is simple to check, using \eqref{defti}, that
 $$\langle B_{i}(\andre{m^*_{n}})\rangle \le \sum_{k=1}^\infty k^{- \alpha}<\infty,$$ for $i \in  \{1, 2\}$.\\
  As for Assumption II) i), we can definitely choose     $m^*_n$ (by considering subsequences) in such a way that it  satisfies
  $$ \sum_{n=1}^\infty 1- \bbP_{m^*_n}(\mathcal{A}(j) \cap  A_\infty) < \infty.$$
 There exists random time $N^*$ such that if $n \ge N^*$ then $\mathcal{A}(j)\cap   A_\infty$ holds $\bbP_{m^*_n}$-a.s..
 We turn to Assumption II ii).  We have
 $$
 \langle B_{1}(m^*_n)\rangle \le \sum_{s=\pi_j(m^*_n)\wedge \pi_{j \oplus 1}(m^*_n)}^{\infty} \frac 1{s^\alpha}
 \le \frac 1{\alpha-1} (\pi_j(m^*_n)\wedge \pi_{j \oplus 1}(m^*_n) - 1)^{1 - \alpha}.
$$
Whenever a  red ball is picked from GUP$(\mathbf{M'}, M'_1 = m^*_n, f_1, f_2)$, the following happens. The process $\mathbf{X}^{\ssup \ell}$ either \andre{jumps} from $j$  to  $j \ominus 1$ or from $j \oplus 2$  to $j \oplus 3$. Before returning to vertex $j$ or  $j \oplus 2$, it must cross again either edge $j \oplus 2$  or edge $j \ominus 1$. 
 In the evaluation of $\langle\Tcal_2(m^*_n)\rangle$ and $ \langle B_{2}(m^*_n)\rangle$   we don't count the transition weights used to jump from $j \oplus 3$ to $j \oplus 2$ and  the ones used to jump from $j \ominus 1$ to $j$. In the evaluation of $\langle \Tcal_1(m^*_n)\rangle$ and $\langle B_{1}(m^*_n)\rangle$ we don't count the transition weights used  to jump from either $j \oplus 1$ to $j$ or from $j \oplus 1$  to $j \oplus 2$. 
Recall that for this GUP we have that the initial state satisfies  $\pi_{j \ominus 1}(m^*_n) \wedge \pi_{j \oplus 2}(m^*_n) \ge (K-1)(\pi_j(m^*_n)\vee \pi_{j \oplus 1}(m^*_n) \vee R_1)$, due to \eqref{condmsta}. Recall that $\alpha >1$. We have
$$
 \langle B_{2}(m^*_n)\rangle \le
 \sum_{s =\pi_{j \ominus 1}(m^*_n) \wedge \pi_{j \oplus 2}(m^*_n)}  \frac 1{s^\alpha}
\le \frac 1{\alpha-1} (\pi_j(m^*_n) \wedge \pi_{j \oplus 1}(m^*_n) - 1)^{1 - \alpha},
$$
notice that on $I^c$,
$$ \langle \Tcal_1(m^*_n)\rangle  \ge    \sum_{s = \pi_j(m^*_n)\vee \pi_{ j\oplus 1}(m^*_n) +1 }  \frac 1{(2s +1+R)^{2\alpha}} \ge   \frac 1{2(2 \alpha -1)} (2\pi_j(m^*_n)\vee \pi_{j \oplus 1}(m^*_n)  + R+2) ^{1 - 2 \alpha}. $$

Finally $\eta_{n} = (\pi_j(m^*_n) \wedge \pi_{j \oplus 1}(m^*_n)+1)^\alpha.$ Hence Assumption II ii) is satisfied on $I^c$.
We turn to Assumption II iii). Denote by $f^{\ssup n}_i(k)$, with $i\in \{1, 2\}$, the reinforcement process for this GUP with initial condition $m^*_n$. We associate $f^{\ssup n}_1$ to the edges $j$ and  $j \oplus 1$, while $f^{\ssup n}_2$ is associated to edges $j \ominus 1$ and $j \oplus 2$.   For $n \ge N^*$, we have
$$
\begin{aligned}
\sum_{k=1}^\infty \frac 1{f^{\ssup n}_1(k)}  \1_{A_{\infty}(n)} &\ge  \sum_{s = \pi_j(m^*_n)\vee \pi_{j \oplus 1}(m^*_n)+1  }  \frac 1{(2s +1+R)^{\alpha}} \1_{A_{\infty}(n)}\\
 &\ge 
   \frac 1{2( \alpha -1)} (2\pi_j(m^*_n)\vee\pi_{j \oplus 1}(m^*_n)  + R+1) ^{1 -  \alpha}\1_{A_{\infty}(n)}
 \end{aligned}  
   $$
whereas 
$$
\begin{aligned}
\sum_{k=1}^\infty \frac 1{f^{\ssup n}_2(k)}  \1_{A_{\infty}(n)} &\le \sum_{s = \pi_{j\ominus 1}(m^*_n)\wedge \pi_{j \oplus 2}(m^*_n) }  \frac 1{s^{\alpha}} \1_{A_{\infty}(n)}\le  \frac 1{\alpha -1} (\pi_{j\ominus 1}(m^*_n)\wedge \pi_{j \oplus 2}(m^*_n)    -1) ^{1 - \alpha}\1_{A_{\infty}(n)}\\
&\le \frac 1{\alpha -1} ((K-1)\pi_{j}(m^*_n)\vee \pi_{j \oplus1}(m^*_n)    -1) ^{1 - \alpha}\1_{A_{\infty}(n)}\\
&\le \frac 1{2(\alpha -1)} (\pi_{j}(m^*_n)\vee \pi_{j \oplus1}(m^*_n) +R +1 ) ^{1 - \alpha}\1_{A_{\infty}(n)},
\end{aligned}
$$
where the last inequality holds for $K$ large enough.  
 Hence, Assumption II iii) holds with  the choice $\andre{b_n(k)} = 1$ for all $k, n \in \N$. Using Theorem~\ref{rubin}, we infer that only finite many red or white balls are picked. This contradicts $\bbP_m(I^c \cap\mathcal{A}(j))>0$. 
\end{enumerate}
Next we combine the two cases \andr{described above, to get our result}.  Notice that in each of the  cases we considered, there is a positive probability  that either $\mathbf{M}$ hits a configuration in $\Pi$ within an  a.s. finite time on $I^c$  or  the event I holds. Denote by $t^*$ the hitting time $\mathcal{X}^c \cup \mathcal{Y}$ by the process $\mathbf{M}$.
\andre{In the reasoning above, we sometimes simplified the notation and fixed $A^*$ to be a particular set of edges. For this reason time $t_1$  appeared in one of the estimates. More generally, $ \max_{i \in V} t(1, i)  + v+ 5R + 4 $ is an apper bound for each of the times listed above, with the exception of the case when the initial configuration is in $\mathcal{X} \cap \mathcal{Y}^c$.  
Set $E_0 = 0$ and }
$$ \andre{E_k \Def (\max_{i \in V}t(1, i) \circ  t^* \circ E_{k-1}) + v+ 5R + 4,}$$
where $t(1, i) \circ  t^* \circ E_{k-1}$ is the hitting time of $i$ after the first hitting time of $\mathcal{X}^c \cup \mathcal{Y} $ which happens by time $E_{k-1}$ on.
\andre{Notice that $\{E_k = \infty\} \subset I$. }
Finally, notice that as $\widetilde{m}\notin \Pi$ in our reasoning above was arbitrary, we have
$$ \andre{\bbP_m \left(\{\exists s \in [E_{k-1}+1,  E_{k}] \colon M_s \in \Pi \} \cup I \;|\; M_s \notin \Pi \; \forall s \le  E_{k-1} \right) \ge c,}$$
where $c >0$. \andre{Using the strong Markov property and the Borel Cantelli Lemma, we have our result.}
\hfill \end{proof}
%Let $\Acal$ be the space of increasing sequences $w_i$ such that $
%\begin{remark}
%The constant $R$ in the previous lemma, is chosen large enough to satisfy
%\begin{equation}\label{chR1}
%\sum_{s = k - R} \frac 1{G(s)} \ge 2 \left(\sum_{s \in [k, \infty)\cap 2\N} \frac 1{G(s)}\right)  \vee \left(\sum_{s \in [k, \infty)\cap (2\N +1)} \frac 1{G(s)} \right),
%\end{equation}
%for all $k > R$. It is possible to find such an $R$.
%\end{remark}

%
%The existence of $S^*_n$ is a consequence of  lemma~\ref{impin}.
%Set $\S'$ as follows
%$$ \S' \Def \{\s \in \S \colon \exists m \in \Pi,  n \in \N_0, \mbox{ such that } \bbP_m(M_n = \s) >0.\}$$
For $m\in \Pi$,  define $\iota(m)$ one of the  indices $k$, chosen uniformly at random among the ones  which satisfy
 \begin{equation}\label{iota}
 \pi_k(m) +R_1 \le \pi_{k \ominus 1}(m) \andre{\wedge} \pi_{k \oplus 1}(m), \mbox{ and } \quad\pi_{j \oplus 1} (\widetilde{m})\neq \pi_{j \ominus 1}(\widetilde{m}).
 \end{equation}
   If $m \notin \Pi$,  set $ \iota(m) = \infty$.
Let the event
$$
\begin{aligned} U \Def \{\iota(M_0) < \infty\} \cap \{X^{\ssup \ell}_{k} &= \iota(M_0), X^{\ssup \ell}_{k+1} = \iota(M_0)  \oplus 1 \;\; \mbox{i.o.}\}\cap\\ &\cap \{X^{\ssup \ell}_k = \iota(M_0), X^{\ssup \ell}_{k+1} =\iota(M_0)\ominus 1 \;\; \mbox{i.o.}\}.
\end{aligned}
$$

\begin{lemma}\label{zeronu}
$
\sup_{m \in \Pi} \bbP_{m}(U) \in \{0, 1\}.
$
\end{lemma}
\begin{proof}
The following is a consequence  of Proposition~\ref{prap} in the Appendix,
$$
\sup_{m \in \S} \bbP_{m}(U) \in \{0, 1\}.
$$
As $\bbP_{m}(U) = 0$ if $m\notin \Pi$, we are done.
%
%%
%%%%The event $U$ can be rewritten  as
%$$ \bigcup_{j \in V} A_{\infty}(\ell, j).$$
\hfill
\end{proof}

%\begin{lemma}\label{ciuno} Recall that $\alpha >1$ is the maximum degree of the polynomial $G$. For any $R_1 \in \N_0$,  there exists a constant $C_1, C_2 \in (0, \infty)$, which depends on $\alpha$, such that  for any $y >1$, $ y \in \N$, we have
%$$ C_1 (y+R_1)^{1- \alpha} \le \sum_{s=y} \frac 1{(s+R_1)^\alpha} \le   C_2(y+R_1)^{1- \alpha}.$$
%\end{lemma}
%\begin{proof}
%\begin{equation}\label{inte}
%\begin{aligned}
% C_2 \int_{y-1}^\infty \frac 1{(x+R_1)^\alpha} {\rm d}x &\ge \sum_{s=y} \frac 1{(s+R_1)^{\alpha}}\ge  \sum_{s=y} \frac 1{G(s+R_1)} \\
%&\ge C \sum_{s=y} \frac 1{(s+R_1)^\alpha}  \ge  C \int_{y}^\infty \frac 1{(x+R_1)^\alpha} {\rm d}x.
%\end{aligned}
%  \end{equation}
%  Finally the  ratio between  right  hand-side and the left hand side of \eqref{inte} equal
%$$\left(\frac{y+R_1}{y -1 + R_1}\right)^{\alpha -1}.$$
%The previous expression is  always larger   than one and uniformly bounded in $y \in \N$ for $y>1$, once $R_1$ and $\alpha$ are fixed.  Set
%$$ C_1^{-1} = \max_{y>1} \left(\frac{y+R_1}{y -1 + R_1}\right)^{\alpha -1},$$
%to get our result.
%\hfill
%\end{proof}
\begin{lemma}\label{lemap}
For any $m \in \S $ we have  $\bbP_{m}(U) = 0$.
\end{lemma}
\begin{proof}
For any vector $m \in \S$,
 define the following shift operator, which   means that we  change the labels of the vertices mapping $j \mapsto j \ominus 1$.  More formally, we denote  the shift operator by $\theta \colon m \mapsto m' \in \S$, where under $\bbP_{m'}$ we have
\andre{ $\vec{X}_0 \mapsto \vec{X}_0\ominus1$}, $ N_{\ell, j}(0) \mapsto N_{\ell, j \ominus 1}(0) $ and finally
 $ T_{\ell, j}(0) \mapsto T_{\ell, j \ominus 1}(0)$.
% \begin{cases}
% m(k) \ominus 1   \qquad \mbox{if  $k \le K$}\\
% m( \ell K + (j \ominus 1 ) +1) \qquad \mbox {if  $k = j +  \ell K$ for some $j \in V$ and $\ell \in [K]$}\\
% m( ((1 + \ell)K+ (j \ominus 1 ) + 1)  \qquad\mbox {if  $k = j +  (1+\ell) K$ for some $j \in V$ and $\ell \in [K]$}
% \end{cases}.
 %$$
%In other words, we have
%   $$
% \begin{aligned}
%   \theta {m} = \Big(m(1), &m(2), \ldots, m(v), m(2v),
%   m(v+1), m(v+2), \ldots \\
%   &m(2v -1), m(3v), m(2v +1), m(2v+2) \ldots m(3v-1)\Big).
%   \end{aligned}
%   $$
%  In words, $\theta_1 a$ is the vector obtained from $a$ by relabelling the vertices as follows. To vertex $j \in V$ is assigned the new label $j \oplus 1$.
For $s \in \N$, define $\theta_1 = \theta$ and  recursively $\theta_s = \theta\theta_{s-1}.$
  We have
\begin{equation}\label{rotat}
  \bbP_{m}(U) = \bbP_{\theta_s m}(U).
  \end{equation}
 This is because the event $U$ does not depend on the actual label of the \andre{index  chosen  uniformly at random from the set of indices satisfying } \eqref{iota}. \andre{It is enough to prove that $\bbP_m(U) = 0$ for $m \in \Pi$.}
 In virtue of \eqref{rotat}, for each  $m \in \Pi$, we  can set, using a proper rotation, $\iota(m)=1$.
Consider $\Sc(m)$.  Suppose that each time that ${\bf X}^{\ssup \ell}$  jumps from $1$ to $0$ or $2$ to $3$  we add one \andre{red} ball  and each time it jumps either from
$1$ to $2$ or from $2$ to $1$ we add one \andre{white} ball. \andre{We show that } this urn evolves like a
GUP(${\bf M^*}, M^*_0 = m,  f_1, f_2$), for some choice  of reinforcement processes $f_i (\cdot)$, $i\in \{1, 2\}$ and Markov process ${\bf M^*}$ on $\S$.
To see this, define
$$\tau_1 \Def \inf\{u \colon u \ge 0, X^{\ssup \ell}_u = 1 \mbox{ or } X^{\ssup \ell}_u = 2\}.$$
Define recursively, for $k \ge 2$,
$$ \tau_k \Def \inf\{u \colon u > \tau_{k -1}, X^{\ssup \ell}_u = 1 \mbox{ or } X^{\ssup \ell}_u = 2\}. $$
 Define  the stochastic processes
\begin{eqnarray*}
 f_1(k) \Def T_{1} (\tau_k), \qquad
 f_2(k) \Def T_{0} (\tau_k)\1_{X_{\tau_k} =1} + T_{2} (\tau_k)\1_{X_{\tau_k} = 2},
\end{eqnarray*}
where recall that $T_s = T_{\ell, s}$.
Let $M^*_{k} \Def M_{\tau_{k+1}}$, for $k\in \N_0$. It is trivial to check that A) and B) from the definition of GUP are both satisfied.
 In this context $U$ is a subset of the event, that in the urn described above, both white and red balls are picked infinitely often.
Suppose that
$$\sup_{m \in \andre{\Pi}} \bbP_{m}(U)>0,$$
and reason by contraddiction.  Recall that we can find a sequence $m'_n \in \Pi$  such that
$$ \lim_{n \ti} \bbP_{m'_n}(U) =1.$$
% Moreover, following the reasoning in  the proof of Proposition
%~\ref{prap}, we can choose the subsequence $m'_n$ such that
%$\bbP_{m'_{n-1}}(M_j = m'_n$ )>0$, for some $ j\in \N$.

%and let
%$$U_1(k) = \{j \colon j \le k, X_{\tau_k } = i, X_{\tau_k  +1} = i \oplus 1 \mbox { or }  X_{\tau_k } = i \oplus 1, X_{\tau_k  +1} = i\}.$$
% Let $U_2(k) = [k] \setminus U_1(k)$. Let $\phi_i(k) = \# U_i$, for $i\in \{1, 2\}$.
 Consider  GUP$(\mathbf{M^*}, M^*_1 = m'_n, f_1, f_2)$, as defined in the proof of Lemma~\ref{zeronu}.
\andre{By taking subsequences, we assume that  $\pi_1(m_n') < \pi_1(m_{n+1}')$}.
 Next we prove that this GUP, if $\sup_{m \in \S} \bbP_m(U) >0$, would satisfy both Assumption I and II, yielding a contraddiction.
 As for Assumption I), it is simple to check, using \eqref{defti}, that
 $$\langle B_{i}(\andre{m_{n}'})\rangle \le \sum_{k=1}^\infty k^{- \alpha}<\infty,$$ for $i \in  \{1, 2\}$.\\
  As for Assumption II) i), we can definitely choose, by considering subsequences,  $m'_n$ satisfying
  $$ \sum_{n=1}^\infty 1- \bbP_{m'_n}(A_\infty) < \infty.$$
 We turn to Assumption II ii).  We have
 $$
 \langle B_{1}(m'_n)\rangle \le \sum_{s=\pi_1(m'_n)}^{\infty} \frac 1{s^\alpha}
 \le \frac 1{\alpha-1} (\pi_1 (m'_n) - 1)^{1 - \alpha}.
$$
Whenever a  red ball is picked from GUP$(\mathbf{M^*}, M^*_1 = m_n', f_1, f_2)$, the following happens. The process $\mathbf{X}^{\ssup \ell}$ either \andre{jumps} from 1 to 0 or from 2 to 3. Before returning to vertex 1 or 2, it must cross again either edge 0 or edge 2.  \andre{In the evaluation of $\langle\Tcal_2(m'_n)\rangle$ and $ \langle B_{2}(m'_n)\rangle$   we don't count the transition weights used to jump from 3 to 2 or the ones used to jump from 0 to 1. }
Recall that for this GUP we have that the initial state is in $\Pi$. Hence  $\pi_{0}(m'_n) \wedge \pi_2(m'_n) \ge \pi_1(m'_n) + \andre{R_1}$,  and $\pi_0(m_n') \neq \pi_2(m_n')$,  and this implies
$$
 \langle B_{2}(m'_n)\rangle \le
 \sum_{s =\pi_0(m'_n) \wedge \pi_2(m'_n)}  \frac 1{s^\alpha}
\le \frac 1{\alpha-1} (\pi_1 (m'_n) - 1)^{1 - \alpha},
$$
notice that on $I^c$,
$$ \langle \Tcal_1(m'_n)\rangle  \ge    \sum_{s = \pi_1(m'_n)+1 }  \frac 1{(s+R)^{2\alpha}} \ge   \frac 1{2 \alpha -1} (\pi_1(m'_n)+R+1) ^{1 - 2 \alpha}. $$

Finally $\eta_{n} = (\pi_1(m'_n))^\alpha.$ Hence Assumption II ii) is satisfied.
We turn to Assumption II iii). Denote by $f^{\ssup n}_i(k)$, with $i\in \{1, 2\}$, the reinforcement process for this GUP with initial condition $m'_n$. We associate $f^{\ssup n}_1$ to the edge 1.   
 As on $m_n'$ we have that $\pi_0(m_n') \neq \pi_2(m_n')$ this implies that
$$  \sum_{k=1}^\infty \frac 1{f^{\ssup n}_2(k)}  \le \sum_{s =\pi_0(m'_n) \wedge \pi_2(m_n')}  \frac 1{s^{\alpha}} \le \frac 1{\alpha-1} (\pi_1(m_n')+R_1-1) ^{1 -  \alpha}, $$
where in the last inequality, we used
$$ \pi_1(m_n') +R_1 \le \pi_{0}(m_n') \andre{\wedge} \pi_{2}(m_n').$$
On the other hand,  we have that
$$ \sum_{k=1}^\infty \frac 1{f^{\ssup n}_1(k)}  \1_{A_{\infty}(n)}   \ge  \sum_{s =\pi_1(m'_n)+1}  \frac 1{(s + R)^{\alpha}} \1_{A_{\infty}(n)}   \ge \frac 1{\alpha-1} (\pi_1(m_n')+R+1) ^{1 - \alpha}\1_{A_{\infty}(n)},$$
%The  inequality comes from $\pi_1 (m'_n) +R \le \pi_0(m'_n) \wedge \pi_2(m'_n).$
  Then,  as  $R_1 \andre{\ge}  R +2$, we have 
\begin{equation}\label{assumptiii}
\begin{aligned}
\sum_{k=1}^\infty \frac 1{f^{\ssup n}_1(k)}  \1_{A_{\infty}(n)} \ge   \sum_{k=1}^\infty \frac 1{f^{\ssup n}_2(k)} \1_{A_{\infty}(n)}. 
\end{aligned}
\end{equation}
 Hence, Assumption II iii) holds with  the choice $\andre{b_n(k)} = 1$ for all $k, n \in \N$,.\\
 Hence, if
$\sup_{m \in \Pi} \bbP_m(U) >0 $ we have that
 Assumption I and II hold for this particular GUP. In virtue of Theorem~\ref{rubin}  and Lemma~\ref{zeronu}, we have a contradiction.~\hfill
\end{proof}

\andre{We denote by $H^*$ the minimum time satisfying  the following. For $u \in \{\iota(M_0), \iota(M_0) \oplus 1\}$,  the condition
$H^* <  t(k, u), t(s, u) < \infty$,   implies that
$$ X^{\ssup \ell}_{t(k, u)+1}= X^{\ssup \ell}_{t(s, u)+1} .$$
%Lemma \ref{lemap} implies that $H^*$ is finite on $I^c$.
In words,  the process $\mathbf{X}^{\ssup \ell}$  after  each visit to either $\iota(M_0)$  or $\iota(M_0) \oplus 1$  made after time $H^*$, it  steps always  in the same direction.   If the direction taken is from $\iota(M_0)$ to $\iota(M_0) \oplus 1$ and from $\iota(M_0) \oplus 1$  to $\iota(M_0)$, then  I holds, as  either the process visits only two vertices  at all large times or  the set $\{\iota(M_0), \iota(M_0) \oplus 1\}$ is visited only finitely often. The alternative, in virtue of Lemma~\ref{lemap}, is that from $\iota(M_0)$ it always jumps to $\iota(M_0) \ominus 1$ and from $\iota(M_0) \oplus 1$ it always jumps to $\iota(M_0) \oplus 2$.
Notice that $H^*$  is not stopping time, and  that $H^*$ can be infinite. On the other hand, if $M_0 \in \Pi$, $H^*<\infty$ on $I^c$, due to  Lemma~\ref{lemap}.}

For $j \in E$, denote by
$$ \Pi_j' \Def \{ m \in \S \colon \pi_j(m) \vee \pi_{j \ominus 1}(m) - R_1 \ge   \pi_j(m) \wedge \pi_{j \ominus 1}(m)\}.$$
\begin{lemma}\label{ccent} Fix $m \in \Pi$,  and $j \in E$. We have
\begin{equation}\label{cent}
 \bbP_m \Big( I \cup \big\{M_t \in \Pi_j',  \mbox{\andre{ i.o.}}\}\Big) =1.
\end{equation}
\end{lemma}
\begin{proof}
To prove \eqref{cent} we reason as follows.  Recall that  $t(k,j)$ is the $k$-th visit   to vertex $j$.  Define $J^+(t) \in \{j \ominus 1, j\}$ as the maximiser of $u \mapsto T_{\ell, u}(t)$, \andre{with $u \in \{j \ominus 1, j\}$.} \andr{In case of equality, we pick  edge $j $}. Define
$$ D_k(j) \Def \bigcap_{s =k}^{k + R_1+R}\big\{X_{t(s,j)+1} = J^+(t(s,k)), \quad t(s,j) < \infty\}\cup \{t(s,j) =  \infty\}.$$
  \andr{Our goal is to  prove that }
\begin{equation}\label{cent1}
\begin{aligned}
 \bbP_m \Big( D_k(j) \mbox{ holds for infinitely many $k$}\Big) = 1.
 \end{aligned}
\end{equation}
In fact, if $ t(s,j) =  \infty$ then $t(u,j) =  \infty$ for all $u \ge s$. The  probability, that  after time $t(s, j)$, $k \le s \le k +R_1 +R$,    the process will traverse the edge $J^+(t(s,k))$, conditionally to the past,  is at least 
$$\frac 1{1 + (1+ R_1 +2R)^\alpha}. $$ 
This is because edge $J^+(t(s,k))$ has an advantage at time $t(s,k)$, in terms of transition probabilities.
This, \andre{together with  Borel-Cantelli Lemma},  proves \eqref{cent1}. The extra $R$ appearing in the definition of $D_k(j)$ is due to the gap between $T_j(n)$ and $N_j(n)$.
%\andre{As $\{t(s,j) =  \infty\}\subset I$, then  using Lemma~\ref{lemap} either} $I$ holds or  after the  random time $\andre{H^*}< \infty$ the following holds.  If $ J^+(t(s, j))$ is traversed at time $t(s, j)+1$, with $ t(s, j) \ge \andre{H^*}$, then
%\begin{equation}\label{101}
%\andre{J^+(t(s+1, j)) }= \andre{J^+(t(s, j)) } \qquad \mbox{if } t(s+1, j) < \infty.
%\end{equation}
%The latter is a consequence of the fact that after time $H^*$, each time the process returns to $j$, it does it through the same edge used in the previous jump from $j$. Roughly, we broke the cycle structure. Combining \eqref{101} with
% \eqref{cent1},  we conclude that  either $I$ holds, or  $\exists t \in \N_0$, such that  $ M_t \in \Pi_j'$. 
\hfill
\end{proof}
%Fix a sequence $m_n'' \in \Pi$ such that $\lim_n \min_{u \in V} \pi_u(m_n'') = \infty$.
\begin{lemma}\label{ainf0int}  Fix $m \in \Pi$ and \andre{$u \in \{1, 2, \ldots, v-1\}$.  Set $ j = \iota(M_0) \oplus u$. } Then
\begin{equation}\label{plm}
\bbP_m\Big(\andre{\{M_0 \in \Pi'_j\}}\cap \{X^{\ssup \ell}_{k} = j, X^{\ssup \ell}_{k+1} = j  \oplus 1 \;\; \mbox{i.o.}\} \cap \{X^{\ssup \ell}_k = j, X^{\ssup \ell}_{k+1} =j\ominus 1 \;\; \mbox{i.o.}\} \cap \{H^*  = 0\}\Big)=0.
\end{equation}
\end{lemma}
\begin{proof}
Denote by $U'$ the event appearing in the left-hand side of \eqref{plm}.
Suppose that $\sup_{m \in \Pi'} \bbP_m(U') > 0$.  This implies that   there exists a sequence $\andre{m_n''} \in \Pi'$, such that
\begin{equation}\label{fig}
  \lim_{n \ti} \bbP_{m_n''}\Big(U'\Big) =1.
\end{equation}
\andre{By the proof of Proposition~\ref{prap}, we can choose $m_n''$ in such a way that the state $m_n''$ communicates with $m_{n+1}''$, i.e. there exists $s$ such that $\bbP_{m_n''}(M_s = m_{n+1}'')>0$. This implies that $\lim_{n \ti} \min_{j \in E} \pi_j (m_n'') = \infty$.}
Consider a sequence of independent $\Sc(m_n'')$, where $m_n''$ satisfies \eqref{fig}. Denote by $j_n = \iota(m_n'') \oplus u$, where $u \in [v-1]$.
In order to model the jumps from vertex $j_n$ we have to consider an independent sequence of GUP$(j_n, m_n'')$.
%Notice that we can assume   that either a.s. \andre{ finite    $X^{\ssup{\ell}}_{t(k,n, j) +1} \in \andre{\{j \ominus 1, j \plus 2\}}$ for all $ k \in \N_0$ or $I$ holds. }

\andre{We check that  this sequence of GUP  }satisfies Assumptions I and II \andre{on $\{H^* = 0\}$}. Assumption I is quite simple. As for Assumption II, choose any sequence which satisfies Assumption II i). \andre{Next we use the fact that on the event \andre{$\{H^* = 0\}$} we have that each time the process $\mathbf{X}^{\ssup \ell}$ jumps from   vertex $j_n$ it goes back to that vertex using the same edge, on $I^c$.}
We have, on $I^c$,
$$  \langle \Tcal_1(m_n'')\rangle  \ge    \sum_{k = \pi_{j - 1}(m_n'')+1 }  \frac 1{(2k+1+R)^{2\alpha}} \ge   \frac 1{(2 \alpha-1)}  (2 \pi_{j-1}(m_n'') + 2 + R) ^{1 - 2 \alpha}, $$
whereas
$$ \langle \Tcal_2(m_n'')\rangle  \ge    \sum_{k = \pi_{j}(m_n'')+1 }  \frac 1{(2k+1+R)^{2\alpha}} \ge \frac 1{(2 \alpha -1 )} (2 \pi_{j}(m_n'') +2 +  R) ^{1 - 2 \alpha}.$$
This is due to the fact that after time $H^*$, roughly speaking, the cycle structure is broken, as explained above.
Moreover,
\begin{eqnarray*}
 \langle B_{1}(m_n'')\rangle &\le& \sum_{k=\pi_{j-1}(m_n'')}^{\infty} \frac 1{k^\alpha
} \le C_1   (\pi_{j-1} (m_n''))^{1 - \alpha},\\
 \langle B_{2}(m_n'')\rangle &\le& \sum_{k=\pi_{j}(m_n'')}^{\infty} \frac 1{k^\alpha
} \le C_1   (\pi_{j} (m_n''))^{1 - \alpha}.
\end{eqnarray*}
Hence Assumption II ii) holds.\\
As for Assumption II iii),  we
pick $b_n(k)$   as follows. It equals $1$ if $\pi_{j}(m_n'') \ge  \pi_{j\ominus 1}(m_n'')+R_1 $ and equals $2$ otherwise. Recall that $1$ is associated to white balls (move to the left) white $2$ with red ones (move to the right). In fact, if $\pi_{j}(m_n'') \ge  \pi_{j\ominus 1}(m_n'') + R_1$, we have
$$ 
\begin{aligned}
\sum_{k=1}^\infty \frac 1{f^{\ssup n}_2(k)}  \1_{A_{\infty}(n)}   &\ge  \sum_{s =\pi_{j\ominus 1}(m''_n)+1}  \frac 1{(2s + R)^{\alpha}} \1_{A_{\infty}(n)}   \ge \frac 1{\alpha-1} (2\pi_{j\ominus 1}(m''_n)+R+1) ^{1 - \alpha}\1_{A_{\infty}(n)}\\
&\ge \frac 1{\alpha-1} (2 \pi_{j \ominus 1}(m_n'')+2R_1-1) ^{1 -  \alpha}\1_{A_{\infty}(n)}\\
 &\ge \sum_{s =\pi_j(m''_n)}  \frac 1{(2s)^{\alpha}} \1_{A_{\infty}(n)} \ge 
\sum_{k=1}^\infty \frac 1{f^{\ssup n}_1(k)}\1_{A_{\infty}(n)}.
\end{aligned}
$$
Similar reasoning, applies for the case $\pi_{j\ominus 1}(m_n'') \ge  \pi_{j}(m_n'')+R_1$.

Hence if we reason by contradiction and assume that both colours (directions) are picked infinitely often with positive probability, then   GUP satisfies both Assumption I and II and Theorem~\ref{rubin} yields a contradiction.
\hfill
\end{proof}

\noindent{\bf Proof of Theorem~\ref{cyclethm}.}
Suppose that   $m$ is a configuration such that $\bbP_m(I^c) >0$.  Lemma \ref{impin} implies that on $I^c$  the process $\mathbf{M}$ hits a configuration in $\Pi$. Moreover on $I^c$,  $U^c$ holds (via Lemma~\eqref{lemap}) and the process $\mathbf{M}$ hits a  configuration in $\Pi_j'$ after $H^*$, in virtue of \ref{ccent}.
In virtue of Lemma \ref{ainf0int}, for each $j \in V$
 there exists a random time $H_j$ and a vertex $v_j \in \{j \ominus 1, j \oplus 1\}$ such that
$X_{t(k, j) +1}^{\ssup \ell} = v_j$ for all  finite $t(k, j) \ge H_j$. This proves that 
\begin{equation}\label{le3G}
\bbP_{m}\Big(I^c \setminus (G_0 \cup G_1 \cup G_2) \Big) = 0\qquad \mbox{ for all $m \in \S$},
\end{equation}
 where
\begin{eqnarray*}
G_0&\Def& \{X^{\ssup \ell}_{k+2} =  X^{\ssup \ell}_{k}, \qquad \mbox{for all  large enough $k \in \N\}$}\\
G_1 &\Def& \{X^{\ssup \ell}_{k+s} = X^{\ssup \ell}_k \oplus s, \qquad \mbox{\andre{for $s \in \N$ and for all  large enough $k \in \N\}$}}\\
G_2 &\Def& \{ X^{\ssup \ell}_{k+s} =  X^{\ssup \ell}_k\ominus s, \qquad \mbox{\andre{for $s \in \N$ and for all  large enough $k \in \N\}$}}.
\end{eqnarray*}
Next, we prove that
\begin{equation}\label{le2G}
  \bbP_{m}(G_1 \cup G_2) =0, \qquad \mbox{for all } m \in \S.
 \end{equation}
Denote by $\andre{\overline{H}}= \max H_j$. Notice that there exist $j$ such that
$\pi_j(M_{\andre{\overline{H}}}) \le \pi_{j \ominus 1}(M_{\andre{\overline{H}}})$.
We have
$$ \frac{T_{\ell, j}(t(k,j))}{T_{\ell, j}(t(k,j)) + T_{\ell, j \ominus 1}(t(k,j))} \le 2/3,$$
for all large $k$, such that  $t(k,j) > \overline{H}$.
This implies that for all sufficiently large $s$, we have 
\begin{equation}\label{stnk}
\bbP_m(G_1 \cap \{\overline{H} \le t(s, j)\}) \le \prod_{k=s}^\infty  \frac{T_{\ell, j}(t(k,j))}{T_{\ell, j}(t(k,j)) + T_{\ell, j \ominus 1}(t(k,j))} \le \prod_{k=s}^\infty \frac 2{3}= 0,
\end{equation}
for all $m \in \S$,   and $ s \in \N$.
\andre{By sending $s \to \infty$, \eqref{stnk}  implies that $\bbP_m(G_1) = 0$.  With a very similar argument we infer $\bbP_m(G_2) = 0$. Hence, \eqref{le3G} combined with \eqref{le2G} implies that $\bbP_m(I^c \setminus G_0) = 0$. This gives a contraddiction, as  $G_0  \subset I$ and we assumed $\bbP_m(I^c)>0$.}
%Next, we reason by contradiction and assume that $\bbP_m(I^c) >0$. We have $\sup_{m \in \S} \bbP_m(I^c) =1$, using proposition~\ref{prap} in the appendix.
%Fix a sequence  $m^*_n$ such that $\bbP_{m^*_n}(I^c) \ge 1 - 2^{-n}$.  Denote by $\mathbf{M}^{\ssup n}$ independent copies of $\mathbf{M}$ with $m^*_n$ as starting point. Set   $\bbP = \bigotimes_n  \bbP_{m^*_n}$. Define
%%$$ S^*_n \Def \inf\{k \in \N_0 \colon M^{\ssup n}_k \in \Pi\}.$$
%On the set $S^*_n = \infty$, define $M^{\ssup n}_k  = u$ where $u$ is a fixed configuration in $\Pi$.
% There exists a random time $N_1$  such that $\{S^*_n < \infty\}$ holds  for all $n \ge N_1$.

Hence,  we proved that $\bbP_m (I^c) = 0$, for all $m\in \S$. Next, we prove that $V'$ can be taken to contain exactly two adjacent  vertices.  We reason again by contradiction.
%\begin{equation}\label{V}
%\bbP_m\Big(\exists j \colon \{X_n^{\ssup \ell}, X_{n+1}^{\ssup \ell}\} \in \{j, j \oplus 1\} \mbox{ for all large $n$}\Big) <1.
%\end{equation}
Fix a vertex  $j$ such that $1< j+1< v-1$.
\begin{equation}\label{contrV}
\sup_{m \in \S} \bbP_m\Big(  \{ X_n^{\ssup \ell} = 0, \mbox{i.o.}\}\cap  \{ X_n^{\ssup \ell} = j+1 , \mbox{i.o.}\}\Big) >0.
\end{equation}
%which contradicts \eqref{V}.

Consider GUP($j$). We proved that after a random time, each time the process $\mathbf{X}^{\ssup \ell}$ makes a jump from $j$, if it returns to $j$ it does it through the same edge used in the jump. Hence after a random time, the GUP behaves like an urn.
% Using \eqref{cent1}, and the fact that  for all large $s$ and $k$,
% we have  $\andre{J^+(t(s, j)) }= \andre{J^+(t(k, j)) }$,
% it is immediate to infer
%\begin{equation}\label{cent3}
% \bbP_m \Big( \{t(s, j) = \infty, \mbox{for some }  j \in \N\} \cup \big\{M_t \in \Pi_{j}',  \mbox{\andre{ i.o.}}\}\Big) >0.
%\end{equation}
Using the similar estimates appearing in the proof of Lemma~\ref{ainf0int} we obtain that $\sup_{m \in \S} \bbP_{m \in \Pi_j'}(A_{j}(\infty)) =0.$  After
a certain random time, each jump from $j-1$ would either always go towards  $j \ominus 2$ or always towards $j$. In the former case, $j$ is visited finitely often, which yields a contradiction. In the other case, the walk we have that the event
$$\Big\{\{X_n^{\ssup \ell}, X_{n+1}^{\ssup \ell}\} \in \{j-1, j\}  \quad \mbox{ eventually}\Big\},$$
holds, yielding another contradiction. The particular choice of the set $\{0, 1, 2, \ldots,  j+1 \}$ does not affect the generality of the result. We just use a relabelling of the vertices and a union bound to get a contradiction. We conclude that the process $\mathbf{X}^{\ssup \ell}$ oscillates between exactly two vertices at all large times.
\hfill \qed
\section{Proof of Theorem~\ref{Examp1}}
We consider the urn defined in Example ~\ref{Example1}, but  with general initial conditions.
Denote by $(N_1(0), N_2(0))\in \N^2$ the initial composition of the urn, i.e. $N_1(0) $ (resp. $N_2(0)$) white (resp. red) balls. In Example~\ref{Example1} the initial composition was assumed to  be $(1,1)$. The evolution of the urn is the same as described in the Example, and the constraint on $g_1(Z_n)$ and $g_2(Z_n)$ becomes
$$ g_i(Z_n) \le \theta \left(N_1(n) + N_2(n)\right),$$
a.s., for all large $n$, $i \in \{1, 2\}$, where the properties of $\theta$ are described in Example \ref{Example1}.
Notice that in virtue of \eqref{prot}, combined with the monotonicity of $\Psi$, we have $\theta(n+2) \le  (1/2) \ln n$ for all large $n$.\\
The process $M_n = (N_1(n), N_2(n), Z_n)$ evolves as an homogeneous markov chain on a countable state space. Set
$ f_i(n) = \Psi(N_1(n) + g_i(Z_n))$, for $i\in \{1, 2\}$.   Both $ f_1(n) $ and $f_2(n)$ are functions of $M_n$. With this representation, it is trivial to see that the urn in Example~\ref{Example1} evolves as GUP($M_n, M_0, f_1, f_2$).
\begin{proposition}\label{utial} Consider an urn with arbitrary initial conditions. The event
$$\{ |N_{1}(n) - N_{2}(n)| \ge (1/2)\floor{\ln n} +1, \quad \mbox{i.o.}\}$$
holds a.s..
\end{proposition}
\begin{proof}
Let $v(n) \in \{1, 2\}$ the  index which  maximizes  $i \mapsto N_i(n)+ g_i(Z_n)$. In case of equality we choose an index uniformly at random.
Define
$$ C_n \Def \left\{N_{v(n)}\Big(n + \floor{\ln n } +1\Big) - N_{v(n)}(n) = \floor{\ln n} +1\right\}.$$
Notice that
$$ \bbP(C_n \;|\; M_{n-1}) \ge  (1+ o(1/\ln n))^{\floor{\ln n} +1}\left(\frac 12\right)^{\floor{\ln n} +1}.$$
Using $\floor{\ln n} +1 \le \log_2 n$ for all large $n$ combined with the second Borel Cantelli lemma, over disjoint  and hence independent blocks, we have
$$\bbP(C_n \;\mbox{ i.o.}) =1.$$
On the other hand,
$$ C_n \subset \left\{\Big|N_{1}\big(n +\floor{\ln n}\big) - N_{2}\big(n+ \floor{\ln n} +1\big)\Big|  \ge \floor{\ln n} +1\right\},$$
as $g_i(Z_n) \le  \theta(n+2) \le (1/2) \ln n$, for all large $n$, ending the proof.
\hfill\end{proof}

\noindent{\bf Proof of Theorem~\ref{Examp1}}.  As $\Psi$ satisfies
$\int_{1}^\infty (1/\Psi(u)) \d u<\infty,$
 it is immediate to see that the GUP  associated to the urn satisfies Assumption I.
Next, we reason by contradiction and suppose that for this GUP, we have that there exists $m$ such that $\alpha(m) >0$.  Then we prove that also Assumption II hold, yielding a contradiction (via Theorem\eqref{rubin}). Using the proof of Proposition~\ref{prap}, we can argue the  existence of a sequence $(m_n)$ of comunicating states in $\S$, i.e. for all $n\in \N_0$,
$$ \bbP_{m_n}(M_j = m_{n+1}) >0, \qquad \mbox{for some } j \in \N,$$
  which satisfies Assumption II i).  We consider a sequence of independent urns. The urns have different initial conditions. The Markov chain associated with the $n$-th urn is denoted by ${\bf M}^{\ssup n}$ and satisfies $M^{\ssup n}_0 = m_n$.
Denote by $(N^{\ssup n}_1(0), N^{\ssup n}_2(0))  \in \N^2$ the initial composition of the urn when $M^{\ssup n}_0 = m_n$.
    Using Proposition~\ref{utial}, we can assume that the initial composition of the urn satisfies
\begin{equation}\label{fre}
  |N^{\ssup n}_1(0) - N^{\ssup n}_2(0)| \ge \floor{\ln  \left(N^{\ssup n}_1(0) + N^{\ssup n}_2(0) \right)}+1.
  \end{equation}
  In fact, for any initial state $m_n$, in virtue of Proposition~\ref{utial},  there exists an a.s. finite stopping time $S^*_n$ when
 $$ |N^{\ssup n}_1(S^*_n) - N^{\ssup n}_2(S^*_n)| \ge  \floor{\ln  \left(N^{\ssup n}_1(0) + N^{\ssup n}_2(0) + S^*_n\right)} +1.$$
   Define
$ m'_n$ one of the elements in  the support of $M_{S^*_n}$, such that $\bbP_{m'_n}(A_{\infty}) \ge \bbP_{m_n}(A_{\infty}).$ The existence of such $m'_n$ is a consequence of the markov property and law of total probability.
 We  use $m'_n$ instead of $m_n$, so we drop $'$ and   assume \eqref{fre} to hold.
 Moreover, we can assume that both sequences $N^{\ssup n}_1(0)$, $ N^{\ssup n}_2(0)$ are strictly increasing in $n$.
Using Proposition~\ref{nicefu}, we can easily argue that Assumption II ii) holds.
Set
$$t_n =  \max \left(N^{\ssup n}_1(0), N^{\ssup n}_2(0)\right).$$
We have $ t_n \in \N$ and $\lim_{n \ti} t_n = \infty$.
Notice that using \eqref{fre}, we have
$$ \max_{i} \langle B_i(m_n) \rangle \ge \sum_{k =t_n- \floor{\ln t_n}) -1} \frac 1{\Psi(k + \theta(k))},$$
whereas
$$  \min_i\sum_{k=1}^\infty \frac 1{f^{\ssup n}_i(k)}  \le \sum_{k =t_n} \frac 1{\Psi(k)}.$$
In virtue of
 Proposition~\ref{spben} we conclude that Assumption II iii) holds by choosing $a_n(k)$ constantly equal to the index in $\{1, 2\}$ corresponding to the color corresponding to the  maximizer  $i \mapsto \langle B_i(m_n) \rangle$.

\section{Appendix}
\begin{proposition}\label{prap} Consider a Markov chain $\mathbf{M}$ on a countable state space $\S$. Let $\bbP_m$ the measure under which $M_0 = m \in \S$, a.s.. For any event $A \in \sigma(\mathbf{M})$, we have $\sup_{m \in  \S} \bbP_m(A) \in \{0,1\}.$
\end{proposition}
\begin{proof}
  Suppose that there exists $m_0 \in \S$ such that   $ \bbP_{m_0}(A) >0$. Under this assumption we prove that $\sup_{m \in \S} \bbP_{m}(A) =1$.
 Let $\Fcal_{n} = \sigma(M_{s},  s \le n)$.
There exists a sequence of events $A_{n}$, such that $A_{n} \in \Fcal_{n}$,
$ \bbP_{m_0}(A_{n} \Delta A) = o(1)$. For any fixed $\eps'>0$, choose $n$ large enough that
$ \bbP_{m_0}(A_{n} \Delta A) < \eps'$ and $\bbP_{m_0}(A_{n}) \ge  (1/2)\bbP_{m_0}(A) >0$. Choose $m$, such  that
$$\bbP_{m_0}\big(A \giv A_{n}  \cap  \{M_{n} = m\}\big)  > \sup_{j} \bbP_{m_0}\big(A \giv A_{n} \cap  \{M_{n} = j\} \big) - \eps',$$
where the supremum in the right-hand side is taken over  the integers $j$ such that $\{M_{n}= j\} \cap A_{n} \neq \emptyset$.
Recall  that $M_{n}$ is a Markov chain, and the future of the GUP given $M_{n}$ is independent of $\Fcal_{n-1}$. Hence as $A_{n} \in \Fcal_{n}$, we have
$$ \bbP_{m_0}\big(A \giv A_{n} \cap \{M_{n} = m\}\big)  = \bbP_m\big(A\big).$$
Moreover,
$$ \mathbb{P}_{m_0}(A \cap A_n) = \bbP_{m_0}(A_n) - \bbP_{m_0}(A_n \setminus A) \ge \bbP_{m_0}(A_n) -\eps',$$
where $\eps'$ was defined above.
Fix $\eps> 0$.
By our choice of $m$, we have
$$
\begin{aligned}
\bbP_m\big(A \big) &= \bbP_{m_0}\big(A \giv A_n \cap\{M_{n} = m\}\big) \ge  \sum_{j \in \S} \bbP_{m_0}\big(A \giv A_{n} \cap \{M_{n}= j\}\big) \bbP_{m_0}(M_{n} = j \giv A_{n}) - \eps'  \\
&= \bbP_{m_0}\big(A \giv  A_{n}\big)  - \eps' = \frac{\bbP_{m_0}\big(A \cap  A_{n}\big)}{\bbP_{m_0}(A_{n})}  - \eps'\ge 1 - \frac{\eps'}{\bbP_{m_0}(A_{n})} - \eps'\\
 &\ge  1 - \frac{2 \eps'}{\bbP_{m_0}(A)} - \eps' \ge 1 - \eps, \\
\end{aligned}
$$
if we choose $\eps'$ small enough. In the inequality before the last one, we used the fact that $\bbP_{m_0}(A_{n}) \ge  (1/2)\bbP_{m_0}(A)>0$.  \hfill
\end{proof}

\noindent{\bf Proof of Proposition~\ref{bmht}}
We first compute
\begin{equation}\label{nonc}
 \mathbb{E}^{x} \Big[ H_{a} \wedge H_0 \1_{H_{a} < H_{0}}\Big]= \frac 13 \frac xa (a^{2} - x^{2}).
\end{equation}
We start with
\begin{equation}\label{nonch}
 \mathbb{E}^{x}\Big[{\rm e}^{-\theta H_{a}} \1_{H_{a}< H_{0}}\Big] = \frac{\sinh(x\sqrt{2 \theta})}{\sinh(a\sqrt{2 \theta})}.
\end{equation}
A reference for the previous formula is, for example, Karatzas-Shreve (second edition) page 100 formula 8.28.
By taking the derivative of the right-hand side of \eqref{nonc1}  we obtain
\begin{equation}\label{nonc1}
\begin{aligned}
&\frac 1{\sinh^{2}(a \sqrt{2 \theta})}\left[\cosh(x \sqrt{2 \theta})\sinh(a\sqrt{2 \theta}) \frac{x}{\sqrt{2 \theta}} - \sinh(x \sqrt{2 \theta})\cosh(a\sqrt{2 \theta}) \frac{a}{\sqrt{2 \theta}}   \right]\\
&\qquad= \frac {a^{2} 2 \theta }{\sinh^{2}(a \sqrt{2 \theta})} \left(\frac 1{a^{2} (2 \theta)^{3/2}}\left[\cosh(x \sqrt{2 \theta})\sinh(a\sqrt{2 \theta}) x - \sinh(x \sqrt{2 \theta})\cosh(a\sqrt{2 \theta}) a   \right]\right).
\end{aligned}
\end{equation}
Use the fact that the expression in the first parenthesis approaches 0 as $\theta \to 0$ and apply De L'Hopital to evaluate the limit in the second parenthesis, which turns out to be the right-hand side of \eqref{nonc}. By taking a second derivative, we get \eqref{noc2}.\hfill \qed
\begin{proposition}\label{nicefu} Suppose that the function $\Psi \colon (0, \infty) \mapsto (0, \infty)$  is increasing and either
\begin{itemize}
\item[a)] $\lim_{x\ti} \Psi(x)/\Psi(x-1)  = \infty $, or
\item[b)] $\Psi$ is  twice  continuously differentiable, with  $\int_{1}^\infty (1/\Psi(u)) {\rm d}u <\infty$,  $\liminf_{x \ti} \Psi'(x) = \infty$,
$\lim_{x\ti} \Psi(x)/\Psi(x-1) $ exists in $[1, \infty)$,
 and
\begin{equation}\label{coms}
\lim_{x \ti} \frac{\Psi''(x) \Psi(x)}{\Psi'(x)^2} >0.
\end{equation}
\end{itemize}
We have
\begin{equation}\label{Psi}
 \liminf_{n \ti} \Psi(n) \frac{\sum_{k=n}^\infty (1/\Psi(k)^2)}{\sum_{k=n}^\infty (1/\Psi(k))} >0.
\end{equation}
\end{proposition}
\begin{proof}
We distinguish the two cases. First assume that  $\lim_{x\ti} \Psi(x)/\Psi(x-1) = \infty.$
In this case, we prove that
\begin{equation}\label{superexp}
\limsup_{n \ti} \Psi(n) \sum_{k=n}^\infty \frac 1{\Psi(k)} \le 2.
\end{equation}
The first step to prove  \eqref{superexp} is to prove that
\begin{equation}\label{supexp1}
\limsup_{n \ti} \Psi(n) \sum_{k=n+1}^\infty  (1/\Psi(k))\le 1.
\end{equation}
To this end, notice that
\begin{equation}\label{supexp2}
\limsup_{n \ti} \Psi(n) \sum_{k=n+1}^\infty  (1/\Psi(k))\le   \limsup_{n \ti} \frac{\int_{n}^\infty (1/\Psi(u)) \d u}{\int_{n}^{n+1} (1/\Psi(u)) \d u}.
\end{equation}
Hence, \eqref{supexp1} is a consequence of De L'Hopital applied to the right hand side of \eqref{supexp2} (with a continuous variable $x$ instead of the discrete $n$) and the assumption $\lim_{x\ti} \Psi(x)/\Psi(x-1) = \infty$, which yield
$$
\lim_{n \ti} \frac{\int_{n}^\infty (1/\Psi(u)) \d u}{\int_{n}^{n+1} (1/\Psi(u)) \d u}=1.
$$
 Hence
$$ \limsup_{n \ti} \Psi(n) \sum_{k=n}^\infty \frac 1{\Psi(k)}  = 1 + \limsup_{n \ti} \Psi(n) \sum_{k=n+1}^\infty \frac 1{\Psi(k)} \le 2,$$
proving \eqref{superexp}.
Finally
$$
\begin{aligned}
 \liminf_{n \ti}  \frac{\Psi(n)\sum_{k=n}^\infty (1/\Psi(k)^2)}{\sum_{k=n}^\infty (1/\Psi(k))}  =  \liminf_{n \ti}  \frac{\Psi(n)^2 \sum_{k=n}^\infty (1/\Psi(k)^2)}{\Psi(n)\sum_{k=n}^\infty (1/\Psi(k))} \ge \left(\limsup_{n \ti} \Psi(n) \sum_{k=n}^\infty \frac 1{\Psi(k)}\right)^{-1}\ge 1/2,
\end{aligned}
$$
proving \eqref{Psi} when \eqref{superexp} holds.

 Next, we move to case b).  Our  assumptions imply that
\begin{equation}\label{laz}
 \lim_{x \ti}  \frac{\int_{x-1}^\infty (1/\Psi(u)) \d u}{\int_{x}^\infty (1/\Psi(u)) \d u}
 \end{equation}
exists in $[1, \infty)$.
Notice that
$$
\liminf_{n \ti} \Psi(n) \frac{\sum_{k=n}^\infty (1/\Psi(k)^2)}{\sum_{k=n}^\infty (1/\Psi(k))} \ge \liminf_{n \ti} \Psi(n) \frac{\sum_{k=n}^\infty (1/\Psi(k)^2)}{\int_{n-1}^\infty (1/\Psi(u)) \d u}.
 $$
 Hence, using \eqref{laz}, we infer that a sufficient condition for  \eqref{Psi} to hold, is that
\begin{equation}\label{goodcondition}
\Psi(n) \sum_{k=n}^\infty (1/\Psi(k)^2) > \eta \int_{n}^\infty (1/\Psi(u)) \d u
\end{equation}
for some $\eta >0$, which is specified below, and for all $n \in \N$. A sufficient condition for \eqref{goodcondition} to hold
is that
\begin{equation}\label{goodc2}
\lim_{x \ti} \frac{ \Psi(x) \int_x^\infty  \frac 1{\Psi(u)^2} \d u}{   \int_x^\infty  \frac 1{\Psi(u)} \d u}  \qquad \mbox{exists and is positive}.
\end{equation}
 Using De L'Hopital in \eqref{goodc2} we require that
\begin{equation}\label{provcond}
 \lim_{x \ti} - \frac{ \Psi'(x) \int_x^\infty  \frac 1{\Psi(u)^2} \d u - (1/\Psi(x))}{   1/(\Psi(x))} >0.
\end{equation}
 \eqref{provcond}  holds if and only if
 $$
 \lim_{x \ti} \frac{\int_x^\infty  \frac 1{\Psi(u)^2} \d u}{1/(\Psi'(x)\Psi(x))} \in [0,1)
 $$
 Using the De L'Hopital one more time,
 $$
 \lim_{x \ti}  \frac{\Psi'(x)^2}{\Psi''(x) \Psi(x) + \Psi'(x)^2} \in [0,1),
 $$
which is a consequence of \eqref{coms}.
\hfill
\end{proof}
\begin{proposition}\label{spben} Suppose that the function $\Psi \colon (0, \infty) \mapsto (0, \infty)$  is   differentiable, with $\Psi'(x)>0$ for all $x$. If $\theta(x)$ is a positive increasing function for which there exists $a\in(0,1/2)$ such that
\begin{equation}\label{prot1}
\frac{\Psi(z + a\ln z )}{\Psi\big(z + \theta(z)\big)} \ge  1 + \frac 1{2z}, \qquad \mbox{for all large  $z$},
\end{equation}
  then
$$ \sum_{k= n-(1/2)\floor{\ln n} -1}^\infty \frac 1{\Psi(k +\theta(k))} \ge \sum_{k= n} \frac 1{\Psi(k)}, \qquad \mbox{for all large n}.$$
%\item[a)] For any constant $R$, there exists $R_1 \in \N$ such that
%$$ \sum_{k= n-R_1}^\infty \frac 1{\Psi(k + R)} \ge \sum_{k= n} \frac 1{\Psi(k)}, \qquad \mbox{for all large n}.$$
%\item[b)]
\end{proposition}
\begin{proof}
 It is enough to prove
 \begin{equation}\label{vuot}
   \int_{x - (1/2)\ln x}^\infty \frac 1{\Psi(u +\theta(u))} \d u \ge \int_{x - 1}^\infty \frac 1{\Psi(u)}\d u, \qquad \mbox{for all large $x$}.
   \end{equation}
In order to prove \eqref{vuot}, we apply a suitable change of variable. Let $h(u) = V^{-1}(u)-1$, where $V(x)=x-(1/2)\ln x$ is smooth and one-to-one on $[1,\infty)$. Since by definition $h(x -(1/2)\ln x) = x-1$, a simple change of variables yields
\begin{equation}\label{vuot1}
 \int_{x -(1/2)\ln x}^\infty \frac 1{\Psi(u +\theta(u))} \d u  =  \int_{x - 1}^\infty \frac {1}{ h'\big(h^{-1}(w)\big)\Psi(h^{-1}(w) + \theta(h^{-1}(w))}  \d w.
\end{equation}
%Now we distinguish the two cases. In case $\theta(x) =R$, we can choose $\beta(x) = R_1$ where $R_1$ is any integer larger than $R+1$, to obtain \eqref{vuot} which implies a).
Hence, a sufficient condition for \eqref{vuot} to hold is
\begin{equation}\label{vuot3}
\frac {1}{ h'\big(h^{-1}(w)\big)\Psi\big(h^{-1}(w) + \theta(h^{-1}(w))\big)} \ge \frac 1{\Psi(w)}, \qquad \mbox{for all large $w$}
\end{equation}
or equivalently
$$\frac {\Psi(h(u))}{\Psi\big(u + \theta(u)\big)} \ge  h'(u), \qquad \mbox{for all large $u$}.$$
To this end we show that for all large $w$ we have, for $a \in (0, 1/2)$, 
\begin{equation}\label{niceq}
h(w) \ge w + a\ln w:=V_a(w),
\end{equation}
or equivalently that
$$x-1\leq V_a(V(x)) = x-(1/2)\ln x+a\ln(x-(1/2)\ln x)$$
or that
$$a\ln\frac{x}{x-(1/2)\ln x}+((1/2)-a)\ln x\geq1,$$
which is clearly true for $x$ large enough.

In the same way we have that, for $x$ large enough,
$$h'(V(x))=\frac1{1-\frac1{2x}}\leq 1+\frac1{2V(x)},\mbox{ i.e. for $u$ large enough},\  h'(u)\leq 1+\frac1{2u}.$$

It is now clear that \eqref{prot1} implies that
$$\frac{\Psi(h(u))}{\Psi\big(u + \theta(u)\big)}\ge \frac{\Psi(u + a\ln u )}{\Psi\big(u + \theta(u)\big)} \ge 1 + \frac 1{2u} \ge h'(u) , \qquad \mbox{for all large $u$}.$$
\hfill
\end{proof}

\noindent{\bf Acknowledgements}. This research was supported under Australian Research Council's Discovery Projects funding scheme (project numbers DP140100559, DP120102728 and DP150103588), by JSPS KAKENHI (grants number 23330109, 24340022, 23654056, 25285102) and  the project RARE-318984 (an FP7 Marie Curie IRSES).

% ******************************************************************
% END OF DOCUMENT
% ******************************************************************

\end{document}